\documentclass[11pt]{article}
\usepackage{amsthm,amsfonts, amsbsy, amssymb,amsmath,graphicx}
\usepackage{graphics}
\usepackage[english]{babel}
\usepackage{graphicx}
\usepackage{lineno}
\usepackage{enumerate}
\usepackage{tikz}
\usepackage{tkz-graph}
\usepackage{tkz-berge}
\usepackage{hyperref}
\usepackage{xcolor}
\usepackage{color}
\usepackage{tikz}
\usetikzlibrary{arrows, automata}
\usepackage{chngcntr}
\counterwithout{figure}{section}

\hypersetup{colorlinks=true}

\hypersetup{colorlinks=true, linkcolor=blue, citecolor=blue,urlcolor=blue}

\title{Edge Coalitions in Graphs}
\author{
{\small Nazli Besharati$^1$\thanks{Corresponding author}\ , Azam Sadat Emadi $^2$, Iman Masoumi$^3$ }\\
{\small $^{1}$Department of Mathematics, Faculty of Sciences}\\
{\small Payame Noor University, Tehran, Iran}\\ {\small $^1$nbesharati@pnu.ac.ir} \\
{\small $^{2}$Department of Mathematics, Faculty of Mathematical Sciences}
\\{\small University of Mazandaran, Babolsar, Iran}\\
{\small $^2$mathemadi2025@gmail.com} \\
{\small $^{3}$ Department of Mathematics, Faculty of Mathematical Sciences}
\\{\small University of Tafresh, Tafresh, Iran}\\{\small $^3$ i.masoumi1359@gmail.com}
}

\newtheorem{theorem}{Theorem}[section]
\newtheorem{corollary}[theorem]{Corollary}

\newtheorem{example}[theorem]{Example}
\newtheorem{observation}[theorem]{Observation}
\newtheorem{proposition}[theorem]{Proposition}

\theoremstyle{definition}
\newtheorem{definition}[theorem]{Definition}
\theoremstyle{remark}

\begin{document}

\maketitle
\begin{abstract}
\noindent
Coalition concepts have been extensively studied in domination theory for vertex sets, whereas their edge counterparts have remained largely unexplored. Motivated by this, we introduce the notions of edge coalition, edge coalition partition, edge coalition number, and edge coalition graph.

We prove that every graph admits an edge coalition partition and establish fundamental properties of these concepts. We derive sharp bounds for the edge coalition number, characterize graphs attaining its extremal values, and determine this parameter for several important graph classes, including complete graphs, complete bipartite graphs, paths, cycles, stars, trees, and unicyclic graphs.

We further introduce the edge coalition graph associated with an edge coalition partition and investigate its structural properties. In particular, we characterize the edge coalition graphs of several graph classes and identify all self-edge coalition graphs. These results extend coalition theory from vertices to edges and provide a foundation for further research on edge coalition structures.
\end{abstract}
 \vspace{1mm}
\section{Introduction}
Community issues are often too complex to be addressed by a single agency or organization, making coalitions an effective strategy for achieving common objectives. Such cooperation arises in education, business, government, and coalition voting under proportional representation systems. Motivated by these applications, coalition structures have attracted considerable attention in graph theory.
Throughout this paper, we restrict our attention to coalitions formed by two distinct groups.

Domination is a central topic in graph theory and has been extensively studied. In \cite{hhhmm1}, Haynes \emph{et al.} introduced the concept of coalition for vertices and further investigated its properties in \cite{hhhmm2,hhhmm3,hhhmm4}. A coalition consists of two disjoint vertex sets, neither of which is a dominating set individually, while their union forms a dominating set. They also introduced coalition partitions and established several fundamental results. More recently, coalition theory has been extended through the introduction of double coalitions in general and regular graphs \cite{HenningMojdeh2025DCG,HenningMojdeh2025DCR}.

Edge domination is a natural counterpart of vertex domination and has also been extensively investigated as one of the fundamental edge-based optimization problems in graph theory \cite{arumugam,Baste2020}. However, coalition concepts have been studied only for vertex sets. This naturally raises the following question: can the notion of coalition be extended from vertices to edges? In this paper, we answer this question in the affirmative. To the best of our knowledge, edge coalition partitions and the associated edge coalition graphs have not previously been studied.

Unlike vertex coalitions, edge coalitions are not a straightforward extension, since adjacency between edges differs fundamentally from adjacency between vertices. Consequently, edge coalition partitions exhibit structural properties that do not follow directly from the corresponding vertex theory.

Motivated by these observations, we introduce the concepts of edge coalition, edge coalition partition, edge coalition number, and edge coalition graph, and investigate their fundamental properties.

Several edge-set partitions have been studied in graph theory and computer science, including proper edge colorings, edge-domatic partitions, and matching partitions~\cite{s}. The concept introduced here provides another natural partition arising from edge domination and coalition structures.

The main contributions of this paper are as follows. We introduce the concepts of edge coalition, edge coalition partition, edge coalition number, and edge coalition graph; prove that every graph admits an edge coalition partition; establish sharp bounds for the edge coalition number; characterize graphs attaining its extremal values; and determine the edge coalition graphs associated with edge coalition partitions for several important graph classes.

The remainder of the paper is organized as follows. Section~2 presents the notation and preliminary results. Section~3 develops the basic theory of edge coalitions. Sections~4 and~5 establish bounds and characterizations of the edge coalition number. Section~6 investigates edge coalition graphs of several graph classes. Section~7 concludes the paper with open problems and future research directions.
%
\section{Preliminaries}

In this section, we recall the notation and terminology used throughout the paper. For undefined notation and terminology, we refer the reader to \cite{Bondy,west}.

The edge degree of an edge $e=uv$, denoted by $deg(e)$, is the number of neighbors of the vertices $u$ and $v$. Equivalently,
$deg(e)=|N(u)\cup N(v)|-2,$  where $N(u)$ and $N(v)$ denote the neighborhoods of $u$ and $v$, respectively. The open neighborhood of an edge $e$ is  $N(e)=\{f\in E(G): f \text{ is adjacent to } e\}.$
Each edge in $N(e)$ is called a neighbor of $e$; thus $deg(e)=|N(e)|$.

Let $G$ be a graph of size $m$. An edge $e$ is called a \emph{full edge} if $deg(e)=m-1$, that is, if $e$ is adjacent to every other edge of $G$. For example, every edge of a star is a full edge.

Let $G=(V,E)$ be a graph. A subset $E_i\subseteq E$ is called a \emph{singleton set} if $|E_i|=1$; otherwise, if $|E_i|\ge2$, it is called a \emph{non-singleton set}.

A set $S\subseteq V(G)$ is a dominating set if every vertex in $V(G)\setminus S$ is adjacent to a vertex of $S$. Similarly, an edge-dominating set is a subset $D\subseteq E(G)$ such that every edge in $E(G)\setminus D$ is adjacent to an edge of $D$. The minimum cardinality of an edge-dominating set is called the \emph{edge domination number} of $G$ and is denoted by $\gamma'(G)$. An edge-dominating set $D$ is \emph{minimal} if no proper subset of $D$ is an edge-dominating set.

An edge-domatic partition of a graph $G$ is a partition of $E(G)$ into edge-dominating sets. For further details, see \cite{arumugam,z,Y}.

The line graph of a graph $G$, denoted by $L(G)$, is the graph whose vertices correspond to the edges of $G$, where two vertices are adjacent if and only if the corresponding edges are adjacent \cite{west}. It is well known that every edge-dominating set of $G$ corresponds to a dominating set of $L(G)$; hence, $\gamma'(G)=\gamma(L(G)).$
It follows immediately that every full edge forms a singleton edge-dominating set.

A unicyclic graph is a connected graph containing exactly one cycle and therefore has the same number of vertices and edges. Examples include paw, pan, and sunlet graphs.

The double star $S(p,q)$, where $p\ge q\ge0$, is the graph obtained from two stars $K_{1,p}$ and $K_{1,q}$ by joining their centers with an edge.
\section{Edge Coalitions}

In this section, we introduce the concept of edge coalitions and investigate their fundamental properties.

\begin{definition}
\label{def1}
Let $G=(V,E)$ be a graph. An edge coalition in $G$ is defined as two disjoint sets of edges $E_1$ and $E_2$, where neither $E_1$ nor $E_2$ is an edge-dominating set, but their union $E_1\cup E_2$ forms an edge-dominating set. We say that $E_1$ and $E_2$ form an edge coalition and refer to them as edge coalition partners.
\end{definition}

\begin{definition}
\label{def2}
An edge coalition partition, henceforth called an $ec$-partition, of a graph $G$ is an edge partition
$\pi=\{E_1,\ldots,E_k\}$ such that every set $E_i$ of $\pi$ is either a singleton edge-dominating set, or is not an edge-dominating set but forms an edge coalition with another set $E_j$ in $\pi$. The edge coalition number, denoted by $EC(G)$, is the maximum order $k$ of an $ec$-partition of $G$, provided that such a partition exists. An $ec$-partition of order $EC(G)$ is called an $EC(G)$-partition.
\end{definition}

Note that if $G$ has no full edges, then no subset $E_i$ in an $ec$-partition can be an edge-dominating set. Consequently, every subset $E_i$ must form an edge coalition with another subset $E_j$ in the partition.

To illustrate these concepts, consider the path  $P_6=(e_1,e_2,e_3,e_4,e_5)$. 
The partition $\pi=\{\{e_1,e_5\}, \{e_2\},\{e_3\}, \{e_4\}\}$ is an $ec$-partition of $P_6$. No set of $\pi$ is an edge dominating set, 
but  the subset  $\{e_1,e_5\}$ and $\{e_2\}$ form an edge coalition,  
the subset $\{e_3\}$ and $\{e_1,e_5\}$ form an edge coalition, and 
the subset $\{e_4\}$ and $\{e_1,e_5\}$ form an edge coalition. 
 Therefore, each set establishes an edge coalition with at least one other set. To show that this is also an $EC(P_6)$-partition, we
note that any larger partition of $E(P_6)$ would necessarily be of order $5$ with each edge residing in a
singleton set, namely, the singleton edge partition of $P_6$.
 Indeed, no two-edge set containing the middle edge $e_3$ forms an edge-dominating set of $P_6$. Hence, the singleton edge partition cannot be an $ec$-partition.
Therefore, $EC(P_6)=4$ and $\pi$ is an $EC(P_6)$-partition.\\
The following example shows that, unlike paths, the singleton edge partition of a graph may itself be an $EC(G)$-partition.
Next, consider the cycle $C_5=(e_1,e_2,e_3,e_4,e_5,e_1)$. 
It is straightforward to observe that the singleton edge partition 
$\pi_1$ of $C_5$ constitutes an $EC(C_5)$-partition. Consequently, the edge coalition number of $C_5$ is given by $EC(C_5)=5$.
\begin{definition}
\label{def3}
Let $G$ be a graph of size $m$ with  an edge set $E=\{e_1, e_2, \ldots, e_m\}$. The
singleton partition, denoted $\pi_1$, of $G$ is the partition of $E$ into $m$ singleton sets, that
is, $\pi_1 = \{\{e_1\}, \{e_2\},\ldots , \{e_m\}\}$.
\end{definition}
\begin{definition}
\label{def4}
Let $\pi=\{E_1,E_2,\ldots,E_k\}$ be an $ec$-partition of a graph $G=(V,E)$. The \emph{edge coalition graph}, denoted by $ECG(G,\pi)$, is the graph whose vertices are in one-to-one correspondence with the sets of $\pi$. Two vertices of $ECG(G,\pi)$ are adjacent if and only if their corresponding sets $E_i$ and $E_j$ form an edge coalition in $G$.
\end{definition}

Note that, in Definition \ref{def4}, we make a slight notational adjustment by allowing $E_i$ to represent both a set in $\pi$ and a vertex in $ECG(G,\pi)$. For clarity and consistency, we maintain this convention throughout the paper, relying on the context to distinguish between the two interpretations.

Moreover, for any graph $G$ and any $EC(G)$-partition $\pi$, there exists a corresponding edge coalition graph $ECG(G,\pi)$ with $EC(G)$ vertices.

\section{Preliminary Results}

In this section, we determine the edge coalition numbers for stars, double stars, paths, and cycles.

Since a coalition partition exists for a graph $H$, and similarly, a coalition partition exists for the line graph $L(G)$, it follows that an edge coalition partition exists for any graph $G$. Therefore, we conclude the following.

\begin{observation}
\label{the-exist}
Every graph $G$ has an $ec$-partition.
\end{observation}

The following values follow immediately from the corresponding results on coalition numbers of line graphs together with the identities.
Since $L(K_{1,n})=K_n$, $L(P_n)=P_{n-1}$, and $L(C_n)=C_n$. Additionally, the line graph $L(S_{p,q})$ is obtained from two complete graphs, $K_p$ and $K_q$, sharing a common vertex. Thus, we have the following observation.

\begin{observation}
\label{obs11}
For any star $K_{1,n}$,
\[
EC(K_{1,n})=n.
\]

For any path $P_n$, Theorem 7 of \cite{hhhmm1} implies that   
\begin{equation*}
EC(P_n)=
\begin{cases}
n-1,& \text{if } n\le5,\\
4,& \text{if } n=6,\\
5,& \text{if } 7\le n\le10,\\
6,& \text{if } n\ge11.
\end{cases}
\end{equation*}

For any cycle $C_n$,  Theorem 8 of \cite{hhhmm1} implies that
\begin{equation*}
EC(C_n)=
\begin{cases}
n,& \text{if } n\le6,\\
5,& \text{if } n=7,\\
6,& \text{if } n\ge8.
\end{cases}
\end{equation*}

For any double star $S_{p,q}$,
\[EC(S_{p,q})=p+q+1.\]
\end{observation}

\begin{example}
Consider the path
$P_{13}=(e_1,e_2,e_3,e_4,e_5,e_6,e_7,e_8,e_9,e_{10},e_{11},e_{12})$
for which we define the partition
$\pi=\{E_1=\{e_1,e_3,e_9,e_{11}\},E_2=\{e_2,e_4,e_8,e_{10},e_{12}\},E_3=\{e_5\},E_4=\{e_6\},E_5=\{e_7\}\}$.

Assign labels to the edges according to the labeling sequence
$(1,2,1,2,3,4,5,2,1,2,1,2)$, where an edge labeled $i$ belongs to the set  $E_i$.
To see that $\pi$ is an $ec$-partition of $P_{13}$, observe that none of the sets in $\pi$ is an edge-dominating set.
 However, the pairs  $(E1,E4)$, $(E2,E3)$, and $(E2,E5)$  form edge coalitions. Hence, every set in $\pi$ forms an edge coalition with another set in the partition.
\end{example}
Since the maximum degree of the line graph satisfies
$\Delta(L(G))\le2\Delta(G)-2$, it follows from Theorem 4 of  \cite{hhhmm1} implies the following observation.

\begin{observation}
\label{obs-del-col}
Let $G$ be a graph with maximum degree $\Delta(G)\ge2$, and let $\pi$ be an $EC(G)$-partition. If $X\in\pi$, then $X$ has at most $2\Delta(G)-1$ edge coalition partners.
\end{observation}
As an immediate consequence of the proof of Observation \ref{obs-del-col}, we deduce the following result.

\begin{corollary}
\label{cor6}
For any connected graph $G=(V,E)$ of order $n\ge2$ and size $m$,
 the edge coalition number satisfies the inequality
\[
1\leq EC(G)\leq m\leq \frac{n(n-1)}{2}.
\]
\end{corollary}

Now, we can characterize the graph $G$ for which
$EC(G)=\frac{n(n-1)}{2}$ holds.

\begin{proposition}
\label{prop1}
For any connected graph $G$ with $n\geq2$,
$EC(G)=\frac{n(n-1)}{2}$ if and only if $G$ is one of the complete graphs in the set
$\{K_2,K_3,K_4,K_5\}$.
\end{proposition}

\begin{proof}

Clearly, the singleton edge partition of the complete graph $K_n$
for $n=2,3,4,5$ is an $EC(K_n)$-partition,  implying that   $EC(K_n)=\frac{n(n-1)}{2}$.

Conversely, if $EC(G)=\frac{n(n-1)}{2}$, then $G$ has the singleton edge partition.
Moreover, since the maximum number of edges in any connected graph of order $n$ is
$\frac{n(n-1)}{2}$, we conclude that  $G\simeq K_n$.

Now we prove that  $2\leq n\leq5$. Let $e_1$ and $e_2$ be two edges.
Then the set $\{e_1,e_2\}$ dominates at most $4n-10$ edges.
On the other hand, it is easy to see that the inequality $4n-10<\frac{n(n-1)}{2}$  holds if and only if  $n\geq6$ or $n\leq3$.
Thus, if   $n\leq3$, the singleton sets are edge-dominating sets.
For $n\geq6$, $G$ does not have  the singleton edge partition,
since  every set   $\{e_1,e_2\}$   does not dominate at least one  vertex of $K_n$.
\end{proof}

Note that the graph $K_2$ is the only graph that attains the lower bound in Corollary \ref{cor6}, whereas the complete graphs
$K_n$ for  $n\in\{3,4,5\}$  attain the upper bound.

\begin{theorem}[\cite{arumugam}]
\label{arumugam}
For any connected graph $G$ of even order $n$,
the edge domination number satisfies $\gamma'(G)=\frac{n}{2}$
if and only if  $G$ is isomorphic to  $K_n$  or  $K_{\frac{n}{2},\frac{n}{2}}$.
\end{theorem}
\begin{proposition}
\label{prop111}
For any complete graph $K_n$ of even order,  $EC(K_n)\geq2(n-1)$.
\end{proposition}

\begin{proof}
Suppose that the vertices of $K_n$ are denoted by  $a_1,a_2,\ldots,a_n$.
According to Theorem \ref{arumugam},  $\gamma'(K_n)=\frac{n}{2}$. 

Moreover, the number of such sets is $\frac{\frac{n(n-1)}{2}}{\frac{n}{2}}=n-1$. 

So, there exist $n-1$ edge-dominating sets of $K_n$.

By dividing each of them into two subsets, an $ec$-partition of $K_n$ can be constructed.
Therefore, $EC(K_n)\geq2(n-1).$  This bound is sharp for $K_4$.
\end{proof}

\begin{proposition}
\label{prop2}
For the complete bipartite graph $K_{r,s}$ with $2\leq r\leq s$, the edge coalition number satisfies the inequality
$EC(K_{r,s})\geq2s$.
\end{proposition}
\begin{proof}
Suppose that the vertices in the partite set with $r$ elements and the vertices in the partite set with $s$ elements are denoted by $a_{1},a_{2},\cdots,a_{r}$ and $b_{1},b_{2},\cdots,b_{s}$, respectively.

Corresponding to each of these $r$ vertices, we have $s$ incident edges whose corresponding sets are
\[
\{a_{1}b_{1},a_{1}b_{2},\cdots,a_{1}b_{s}\},
\;
\{a_{2}b_{1},a_{2}b_{2},\cdots,a_{2}b_{s}\},
\;
\cdots,
\;
\{a_{r}b_{1},a_{r}b_{2},\cdots,a_{r}b_{s}\}.
\]
We consider
$A=\{X_1,X_2,\cdots,X_s\}$,
where
$X_i=\{a_{1}b_{i},a_{2}b_{i},\cdots,a_{r-1}b{i},a_{r}b{i}\}$.
It is easy to see that each  $X_i\in A$  is a minimal edge-dominating set of  $K_{r,s}$.
Clearly, the number of edge-dominating sets in $K_{r,s}$ is $r$.
Since each $X_i$ satisfies
$|X_i|=r$,  dividing each $X_i$ into two non-empty subsets yields an $ec$-partition of  $K_{r,s}$, such as $\pi=\{X_{1,1}, X_{1,2}, X_{2,1},X_{2,2},\cdots, X_{s,1}, X_{s,2} \}$ that for every $1\leq i \leq s$, $X_{i,1}=\{a_{1}b_{i}\}$ and $X_{i,2}=\{a_{2}b_{i},\cdots,a_{r-1}b_{i},a_{r}b_{i}\}$.  Since $r\ge2$, no singleton edge constitutes an edge-dominating set. Thus,  $EC(K_{r,s})\ge2s.$ 
This bound is sharp for $K_{2,2}$.
\end{proof}
\section{Bounds on $EC(G)$ and Graphs with Small $EC(G)$}

In this section, we establish bounds on the edge coalition number of a graph and characterize the graphs $G$ for which
$EC(G)\in\{1,2\}$.

\subsection{Bounds on $EC(G)$}

The proof of the following observation is straightforward; therefore, it is omitted.

\begin{observation}
\label{obs1}
If $G$ is a graph with no isolated edges and no full edges, then
\[3\le2\gamma'(G)-1\le EC(G).\]
\end{observation}
\begin{proposition}
\label{obs2}
If $G\neq K_n$ is a graph of order $n$ with $k$ vertices of degree $n-1$,
then $kn-\frac{k(k+1)}{2}\leq EC(G)$.
Moreover, this bound is sharp.
\end{proposition}
\begin{proof}
Let $v_i$ ($1\le i\le k$) be the vertices of degree $n-1$, and let $v_i$ for $i>k$ be the remaining vertices of $G$.

Consider the edges $v_iv_j$, where $1<j\le n$ and $j>i$, that are incident to $v_i$ for $1\le i\le k$.
Define
\[
A=E(G)\setminus \bigcup_{i=1}^{k}
\left(
\bigcup_{\substack{j=1\\ j\neq i}}^{n}
\{v_iv_j\}
\right).
\]

Then, the set $E_i=A\cup\{v_iv_j:1\le i\le k,\;1<j\le n,\;j>i\}$ 
forms an $ec$-partition. Consequently,   $kn-\frac{k(k+1)}{2}\leq EC(G).$

To verify the sharpness of this bound, consider the path $P_3$, since
\[EC(P_3)=2=3-\frac{1(1+1)}{2}.\]
\end{proof}
\begin{theorem}
\label{the-bound}
If $G$ is a graph with no full edge and minimum degree $\delta(G)\geq 1$, then $1+\delta(G)\leq EC(G)$.
\end{theorem}

\begin{proof}
Let $G$ be a graph of order $n$ with $1\leq \delta(G)\leq n-2$, and let $v$ be a vertex in $G$ with degree
$deg(v)=\delta(G)=k$.

Let $N_E(v)=\{e_1,e_2,\cdots,e_k\}$ be the set of edges incident to $v$.

Since $G$ has no full edge, $E(G)\setminus N[e_1]\neq\emptyset$.   Indeed, if $E(G)\setminus N[e_1]=\emptyset$, then every edge of $G$ is either equal to $e_1$ or adjacent to $e_1$. Hence, $e_1$ is a full edge, contradicting the hypothesis.

We define an edge partition $\pi$ of $G$ as follows:
$E_i=\{e_i\}$, for $1\leq i\leq k$, and  $E_{k+1}=E(G)\setminus N[e_1]$.

To see that $\pi$ is an $ec$-partition of $G$, observe that $E_{k+1}$ does not dominate $e_1$.

Moreover, since $G$ has no full edge, no singleton set edge dominates $G$. Hence, no set in $\pi$ is an edge-dominating set of $G$.

We now consider the following two cases.\\

\emph{i)} Suppose $k=1$, then $\pi=\{E_1,E_2\}$. It is clear that $\pi$ forms an $ec$-partition.\\

\emph{ii)} Suppose $k>1$, then $\pi$ is still an $ec$-partition because $E_1\cup E_{k+1}$ is an edge-dominating set of $G$.

Moreover, for each $2\leq i\leq k$, the set $E_i\cup E_{k+1}$ is also an edge-dominating set of $G$.
We claim that  every edge in $E_i$   $2\leq i\leq k$, is adjacent to at least one edge in $E_{k+1}$, assume, for contradiction, that there exists an edge $e_j\in E_i$ that is not adjacent to any edge in $E_{k+1}$.

In this case, $e_j$ would be a pendant edge, meaning that there exists a vertex 
$v_j$ incident to this edge with ${deg(v_j)=1}$, which contradicts the assumption that $k=\delta(G)>1$.
\end{proof}
\subsection{Graphs with Small Values of the Edge Coalition Number}
We characterize graphs with edge coalition numbers at most three.
\begin{theorem}
\label{the-cases}
 Let $G$ be a connected graph. Then \\
\emph{1}. $EC(G)=1$ if and only if $G = K_2$. \\
\emph{2}. $EC(G)=2$ if and only if  $G\in \{P_3, \overline{C_4}\} $.
\end{theorem}
\begin{proof}
(1) It is straightforward.\\
(2) If $G=\overline{C_4}$  or $G=P_3$, then clearly edge singleton sets will form the only edge coalition 
partition of graph $G$ and so we have $EC(G)=2$.
Conversely, assume that $EC(G)=2$. By (1), $G$ must have order $n\geq 3$. 
If $n=3$, then $G=P_3$, since $EC(K_3)=3$. If $n=4$, the only graph with $EC(G)=2$ is $\overline{C_4}$.
Indeed, every graph of order four, other than $\overline{C_4}$, admits an $ec$-partition with three sets.
Therefore,  the result holds.
\end{proof}
\begin{theorem}
\label{the-char-3}
Let $G$ be a connected graph. Then $EC(G)=3$ if and only if  $G\in \{C_3, P_4, K_{1,3}\}$.
\end{theorem}
\begin{proof}
If $G=C_3$, $G=P_4$ or $G=K_{1,3}$, then clearly the singleton edge partition forms an 
edge coalition partition of $G$, and hence $EC(G)=3$.
Conversely, assume that $EC(G)=3$ and let $G$ be a connected graph with $E(G)=e_1,\cdots ,e_t$.
We claim that $t\leq 3$. Suppose, for a contradiction, that $t\geq 4$. 
We consider the following two cases.

\emph{Case 1)} If $t\geq 5$. Let $A=E(G)\setminus \{e_1,e_2,e_3,e_4\}$, $X=\{\{e_1\}\cup A,\{e_2\},\{e_3\},\{e_4\}\}$ be an edge partition of $G$.
None of the sets $\{e_2\},\{e_3\}$ and $\{e_4\}$ is a singleton edge-dominating set in $G$. 
If the set $\{e_1\}\cup A$ is not  edge-dominating in $G$, then $EC(G)\geq 4$, which contradicts our assumption. 
So, suppose $\{e_1\}\cup A$ is an edge-dominating set. There is an edge $e_i, ~2\leq i\leq4$ that is not dominated by by $e_1$. 
Let $B$ be the edges in $A$ adjacent to $e_i$. 
Since $\{e_1\}\cup A$ is an edge-dominating set and $e_i$ is not adjacent to $e_1$, at least one edge of $A$ must be adjacent to $e_i$. Hence, $B\neq\emptyset$. 

Let $C=\{e_j |~ 2 \leq j \leq 4, j \neq i\}$.
We consider the edge partition $X=\{(A \setminus B)\cup \{e_1\}, \{e_i\}, B, C\}$. 
If $C$ is not an edge dominating set, then $X$ forms an $EC(G)$-partition, since each
of $\{e_i\}$, $C$ and $B$ establishes an edge coalition with $(A\setminus B)\bigcup \{e_1\}$. 
If $C$ is an edge dominating set, then  the edge partition $X=\{(A\setminus B)\bigcup \{e_1\}, \{e_i\}, B, C_1, C_2\}$, such that $C_1$ and $C_2$ are subsets of $C$ that $C=C_1\cup C_2$. So, $EC(G)\geq 4$.

\emph{case 2)} If $t=4$, then we make use of family $\Omega$ depicted as follows,  so as to give the characterization of all connected graphs for which $t=4$.

(1) The cycle graph $C_5$  with at most $5$ chords.

(2) Graphs $G$  obtained from  $C_3$ by coinciding the center of two stars to at least two vertices of $C_3$. Or Graphs $G$ obtained from  $C_3$ by coinciding one center of double star $S_{p,q}$ to one vertex of $C_3$, like $W_1$ and $W_2$ in Figure~\ref{f1}. 
Or  Graphs $G$  obtained from  $C_3$ by coinciding the one leaf of at least one star to exact one vertex of $C_3$.

(3) The graphs $G$  obtained from cycle $C_4=abcda$ by coinciding the center of one star $K_{1,n}$ to one of vertices of $C_4$ or only by
coinciding the centers of two stars $K_{1,m}, K_{1,n}$ to the only two vertices  $a,c$ or $b,d$ of $C_4$, like $W_3$ and $W_7$  in Figure~\ref{f1}. 

(4) The graphs $G$  obtained from $K_4-e$  with four vertices $x,y,z,t$ and edge set $\{xy,yz,zt,tx,xz\}$ by coinciding the center of one star $K_{1,n}$ to one of the vertices of $K_4-e$  or only by
coinciding the centers of two stars $K_{1,m}, K_{1,n}$ to the  two vertices  $x,z$, like $W_4$, $W_5$ and $W_6$  in Figure~\ref{f1}.

(5) Graphs $G$  obtained from at least two  $C_3$ with a common vertex by coinciding the center of the star $K_{1,n}$ to this common vertex and by coinciding the one leaf of at least one star $K_{1,m}$ to this common vertex, like $W_9$ in Figure~\ref{f1}.

(6) The graphs $G$  obtained from $K_4$ by coinciding the center of a star $K_{1,n}$ to only one of the vertices of $K_4$, like $W_8$  in Figure~\ref{f1}.

(7) The graphs $G$  obtained from $m\ge 3$ stars $K_{1,n_i}$, ($1\le i\le m$) by coinciding the one leaf of each  $m-1$ stars $K_{1,n_i}$ and the center of another star, like $W_{10}$  in Figure~\ref{f1}.

(8) The graph $G$  obtained from cycle $C_4=abcda$ by coinciding the centers of two star $K_{1,m}, K_{1,n}$ to the  two vertices  $b,d$ of $C_4$ and   make adjacent some independent vertices to the vertices $b,d$, like $W_{11}$ in Figure~\ref{f1}.

(9) The graph $G$  obtained from cycle $K_4-e$ by coinciding the centers of two stars $K_{1,m}, K_{1,n}$ to the  two vertices of degree $3$ of $K_4-e$  and make adjacent some independent vertices to the two vertices of degree $3$ of $K_4-e$, like $W_{12}$  in Figure~\ref{f1}.

 Thus, $t=4$ if and only if $G\in \Omega$. Also it is observed that for any graph $G$ in $\Omega$, $EC(G)\geq 4$.
Thus, in both cases, we arrive at a contradiction. So, $EC(G)=3$ if and only if $t=3$, it means, $G=P_4$, $G=C_3$ and $G=K_{1,3}$.
\end{proof}
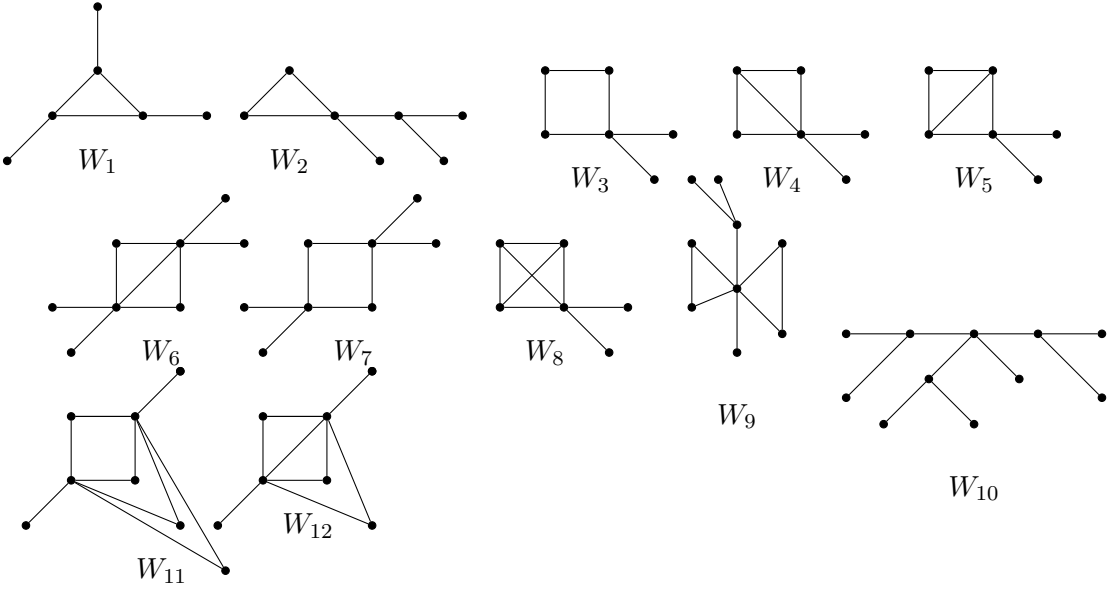
\begin{figure}[!t]
 \begin{tikzpicture}[node distance=0.85cm, inner sep=0pt]
 \node[state, minimum size=1mm] (I_23)[fill=black] {$$};
 \node[state, minimum size=1mm] (I_24)[below right of=I_23, fill=black] {$$};
  \node[state, minimum size=1mm] (I_25)[below left of=I_23, fill=black] {$$};
  \node[state, minimum size=1mm] (I_26)[above of=I_23, fill=black] {$$};
  \node[state, minimum size=1mm] (I_27)[right of=I_24, fill=black] {$$};
    \node[state, minimum size=1mm] (I_100)[below left of=I_25, fill=black] {$$};
  \node(L)[below right of=I_25] { $\small{W_{1}}$};
  \path (I_23) edge (I_24);
 \path (I_23) edge (I_25);
\path (I_24) edge (I_25);
\path (I_26) edge (I_23);
\path (I_27) edge (I_24);
\path (I_100) edge (I_25);

 \node[state, minimum size=0mm] (H_3)[right of=I_23, fill=white] {$$};
  \node[state, minimum size=0mm] (H_4)[right of=H_3, fill=white] {$$};
 \node[state, minimum size=1mm] (I_28)[right of=H_4, fill=black] {$$};
 \node[state, minimum size=1mm] (I_29)[below right of=I_28, fill=black] {$$};
  \node[state, minimum size=1mm] (I_30)[below left of=I_28, fill=black] {$$};
  \node[state, minimum size=1mm] (I_31)[right of=I_29, fill=black] {$$};
   \node[state, minimum size=1mm] (I_101)[below right of=I_31, fill=black] {$$};
  \node[state, minimum size=1mm] (I_32)[right of=I_31, fill=black] {$$};
    \node[state, minimum size=1mm] (I_102)[below right of=I_29, fill=black] {$$};
  \node(L)[below right of=I_30] { $\small{W_{2}}$};
 \path (I_28) edge (I_29);
 \path (I_28) edge (I_30);
\path (I_29) edge (I_30);
\path (I_29) edge (I_31);
\path (I_31) edge (I_32);
\path (I_101) edge (I_31);
\path (I_102) edge (I_29);

 \node[state, minimum size=0mm] (H_5)[right of=I_28, fill=white] {$$};
  \node[state, minimum size=0mm] (H_6)[right of=H_5, fill=white] {$$};
  \node[state, minimum size=0mm] (H_7)[right of=H_6, fill=white] {$$};
 \node[state, minimum size=1mm] (I_33)[right of=H_7, fill=black] {$$};
 \node[state, minimum size=1mm] (I_34)[right of=I_33, fill=black] {$$};
  \node[state, minimum size=1mm] (I_35)[below of=I_33, fill=black] {$$};
  \node[state, minimum size=1mm] (I_36)[below of=I_34, fill=black] {$$};
  \node[state, minimum size=1mm] (I_37)[right of=I_36, fill=black] {$$};
  \node[state, minimum size=1mm] (I_38)[below right of=I_36, fill=black] {$$};
  \node(L)[below right of=I_35] { $\small{W_{3}}$};
 \path (I_33) edge (I_34);
 \path (I_33) edge (I_35);
\path (I_34) edge (I_36);
\path (I_35) edge (I_36);
\path (I_37) edge (I_36);
\path (I_38) edge (I_36);

 \node[state, minimum size=0mm] (H_8)[right of=I_34, fill=white] {$$};
 \node[state, minimum size=1mm] (I_39)[right of=H_8, fill=black] {$$};
 \node[state, minimum size=1mm] (I_40)[right of=I_39, fill=black] {$$};
  \node[state, minimum size=1mm] (I_41)[below of=I_39, fill=black] {$$};
  \node[state, minimum size=1mm] (I_42)[below of=I_40, fill=black] {$$};
  \node[state, minimum size=1mm] (I_43)[right of=I_42, fill=black] {$$};
  \node[state, minimum size=1mm] (I_44)[below right of=I_42, fill=black] {$$};
  \node(S)[below right of=I_41] { $\small{W_{4}}$};
  \path (I_39) edge (I_40);
 \path (I_39) edge (I_41);
\path (I_42) edge (I_40);
\path (I_42) edge (I_41);
\path (I_43) edge (I_42);
\path (I_44) edge (I_42);
\path (I_42) edge (I_39);

\node[state, minimum size=0mm] (H_9)[right of=I_40, fill=white] {$$};
 \node[state, minimum size=1mm] (I_45)[right of=H_9, fill=black] {$$};
 \node[state, minimum size=1mm] (I_46)[right of=I_45, fill=black] {$$};
  \node[state, minimum size=1mm] (I_47)[below of=I_45, fill=black] {$$};
  \node[state, minimum size=1mm] (I_48)[below of=I_46, fill=black] {$$};
  \node[state, minimum size=1mm] (I_49)[right of=I_48, fill=black] {$$};
  \node[state, minimum size=1mm] (I_50)[below right of=I_48, fill=black] {$$};
  \node(L)[below right of=I_47] { $\small{W_{5}}$};
  \path (I_45) edge (I_46);
 \path (I_47) edge (I_45);
\path (I_46) edge (I_48);
\path (I_49) edge (I_48);
\path (I_50) edge (I_48);
\path (I_47) edge (I_48);
\path (I_46) edge (I_47);

\node[state, minimum size=0mm] (H_10)[below of=I_25, fill=white] {$$};
\node[state, minimum size=0mm] (H_11)[below of=H_10, fill=white] {$$};
 \node[state, minimum size=1mm] (I_51)[right of=H_11, fill=black] {$$};
 \node[state, minimum size=1mm] (I_52)[right of=I_51, fill=black] {$$};
  \node[state, minimum size=1mm] (I_53)[below of=I_51, fill=black] {$$};
  \node[state, minimum size=1mm] (I_54)[below of=I_52, fill=black] {$$};
  \node[state, minimum size=1mm] (I_55)[left of=I_53, fill=black] {$$};
  \node[state, minimum size=1mm] (I_56)[below left of=I_53, fill=black] {$$};
   \node[state, minimum size=1mm] (I_57)[right of=I_52, fill=black] {$$};
  \node[state, minimum size=1mm] (I_58)[above right of=I_52, fill=black] {$$};

  \node(L)[below right of=I_53] { $\small{W_{6}}$};
  \path (I_51) edge (I_52);
 \path (I_51) edge (I_53);
\path (I_52) edge (I_54);
\path (I_53) edge (I_54);
\path (I_55) edge (I_53);
\path (I_56) edge (I_53);
\path (I_57) edge (I_52);
\path (I_58) edge (I_52);
\path (I_52) edge (I_53);

\node[state, minimum size=0mm] (H_12)[right of=I_52, fill=white] {$$};
 \node[state, minimum size=1mm] (I_59)[right of=H_12, fill=black] {$$};
 \node[state, minimum size=1mm] (I_60)[right of=I_59, fill=black] {$$};
  \node[state, minimum size=1mm] (I_61)[below of=I_59, fill=black] {$$};
  \node[state, minimum size=1mm] (I_62)[below of=I_60, fill=black] {$$};
  \node[state, minimum size=1mm] (I_63)[right of=I_60, fill=black] {$$};
  \node[state, minimum size=1mm] (I_64)[above right of=I_60, fill=black] {$$};
   \node[state, minimum size=1mm] (I_65)[left of=I_61, fill=black] {$$};
  \node[state, minimum size=1mm] (I_66)[below left of=I_61, fill=black] {$$};

  \node(L)[below right of=I_61] { $\small{W_{7}}$};
  \path (I_59) edge (I_60);
 \path (I_59) edge (I_61);
\path (I_60) edge (I_62);
\path (I_61) edge (I_62);
\path (I_63) edge (I_60);
\path (I_64) edge (I_60);
\path (I_65) edge (I_61);
\path (I_66) edge (I_61);

\node[state, minimum size=0mm] (H_13)[right of=I_60, fill=white] {$$};
 \node[state, minimum size=1mm] (I_67)[right of=H_13, fill=black] {$$};
 \node[state, minimum size=1mm] (I_68)[right of=I_67, fill=black] {$$};
  \node[state, minimum size=1mm] (I_69)[below of=I_67, fill=black] {$$};
  \node[state, minimum size=1mm] (I_70)[below of=I_68, fill=black] {$$};
  \node[state, minimum size=1mm] (I_71)[right of=I_70, fill=black] {$$};
  \node[state, minimum size=1mm] (I_72)[below right of=I_70, fill=black] {$$};

  \node(L)[below right of=I_69] { $\small{W_{8}}$};
  \path (I_67) edge (I_68);
 \path (I_67) edge (I_69);
\path (I_68) edge (I_70);
\path (I_70) edge (I_69);
\path (I_71) edge (I_70);
\path (I_72) edge (I_70);
\path (I_69) edge (I_68);
\path (I_70) edge (I_67);

\node[state, minimum size=0mm] (H_14)[right of=I_68, fill=white] {$$};
 \node[state, minimum size=1mm] (I_73)[right of=H_14, fill=black] {$$};
 \node[state, minimum size=1mm] (I_74)[below of=I_73, fill=black] {$$};
  \node[state, minimum size=1mm] (I_75)[below right of=I_73, fill=black] {$$};
  \node[state, minimum size=1mm] (I_76)[above right of=I_75, fill=black] {$$};
  \node[state, minimum size=1mm] (I_77)[below right of=I_75, fill=black] {$$};
  \node[state, minimum size=1mm] (I_78)[below of=I_75, fill=black] {$$};
 \node[state, minimum size=1mm] (I_79)[above of=I_75, fill=black] {$$};
 \node[state, minimum size=1mm] (I_80)[left of=S, fill=black] {$$};
 \node[state, minimum size=1mm] (I_81)[above left of=I_79, fill=black] {$$};

  \node(L)[below of=I_78] { $\small{W_{9}}$};
  \path (I_73) edge (I_74);
 \path (I_73) edge (I_75);
\path (I_74) edge (I_75);
\path (I_75) edge (I_76);
\path (I_75) edge (I_77);
\path (I_76) edge (I_77);
\path (I_78) edge (I_75);
\path (I_79) edge (I_75);
\path (I_80) edge (I_79);
\path (I_81) edge (I_79);

 \node[state, minimum size=1mm] (I_15)[right of=I_77, fill=black] {$$};
 \node[state, minimum size=1mm] (I_16)[right of=I_15, fill=black] {$$};
  \node[state, minimum size=1mm] (I_17)[right of=I_16, fill=black] {$$};
  \node[state, minimum size=1mm] (I_18)[right of=I_17, fill=black] {$$};
  \node[state, minimum size=1mm] (I_19)[right of=I_18, fill=black] {$$};
   \node[state, minimum size=1mm] (I_20)[below of=I_15, fill=black] {$$};
  \node[state, minimum size=1mm] (I_21)[below of=I_19, fill=black] {$$};
   \node[state, minimum size=1mm] (I_22)[below left of=I_17, fill=black] {$$};
     \node[state, minimum size=1mm] (I_23)[below right of=I_17, fill=black] {$$};
  \node[state, minimum size=1mm] (I_24)[below left of=I_22, fill=black] {$$};
     \node[state, minimum size=1mm] (I_25)[below right of=I_22, fill=black] {$$};
    \node(L)[below of=I_25] { $\small{W_{10}}$};

  \path (I_15) edge (I_16);
  \path (I_16) edge (I_17);
 \path (I_17) edge (I_18);
\path (I_18) edge (I_19);
\path (I_20) edge (I_16);
\path (I_21) edge (I_18);
\path (I_22) edge (I_17);
\path (I_23) edge (I_17);
\path (I_24) edge (I_22);
\path (I_25) edge (I_22);

\node[state, minimum size=1mm] (I_82)[below of=I_56, fill=black] {$$};
 \node[state, minimum size=1mm] (I_83)[right of=I_82, fill=black] {$$};
  \node[state, minimum size=1mm] (I_84)[below of=I_82, fill=black] {$$};
  \node[state, minimum size=1mm] (I_85)[below of=I_83, fill=black] {$$};
  \node[state, minimum size=1mm] (I_86)[below right of=I_85, fill=black] {$$};
   \node[state, minimum size=1mm] (I_87)[above right of=I_83, fill=black] {$$};
  \node[state, minimum size=1mm] (I_88)[below left of=I_84, fill=black] {$$};
   \node[state, minimum size=1mm] (I_89)[below right of=I_86, fill=black] {$$};
  \node[state, minimum size=1mm] (I_90)[above right of=I_83, fill=black] {$$};
    \node(L)[left of=I_89] { $\small{W_{11}}$};

  \path (I_82) edge (I_83);
  \path (I_84) edge (I_82);
 \path (I_85) edge (I_83);
\path (I_82) edge (I_83);
\path (I_86) edge (I_83);
\path (I_86) edge (I_84);
\path (I_84) edge (I_85);
\path (I_88) edge (I_84);
\path (I_89) edge (I_83);
\path (I_89) edge (I_84);
\path (I_90) edge (I_83);

\node[state, minimum size=0mm] (T)[right of=I_83, fill=white] {$$};
\node[state, minimum size=1mm] (I_91)[right of=T, fill=black] {$$};
 \node[state, minimum size=1mm] (I_92)[right of=I_91, fill=black] {$$};
  \node[state, minimum size=1mm] (I_93)[below of=I_91, fill=black] {$$};
  \node[state, minimum size=1mm] (I_94)[below of=I_92, fill=black] {$$};

   \node[state, minimum size=1mm] (I_96)[above right of=I_92, fill=black] {$$};
  \node[state, minimum size=1mm] (I_97)[below left of=I_93, fill=black] {$$};
   \node[state, minimum size=1mm] (I_98)[below right of=I_94, fill=black] {$$};
  \node[state, minimum size=1mm] (I_99)[above right of=I_92, fill=black] {$$};
\node(L)[left of=I_98] { $\small{W_{12}}$};

  \path (I_91) edge (I_92);
  \path (I_93) edge (I_91);
 \path (I_94) edge (I_92);
\path (I_91) edge (I_92);

\path (I_93) edge (I_94);
\path (I_97) edge (I_93);
\path (I_98) edge (I_92);
\path (I_98) edge (I_93);
\path (I_99) edge (I_92);
\path (I_92) edge (I_93);
\end{tikzpicture}
\begin{center}
\caption {Some graphs in $\Omega $ }
\label{f1}
\end{center}
\end{figure}
\section{Graphs with Maximum Edge Coalition Number}
In this section, we study graphs $G$ for which the edge coalition number equals the size of $G$.

\subsection{Unicyclic graphs}

Let $G$ be a connected unicyclic graph of order $n$ and size $m$. Recall that every connected unicyclic graph of order $n$ has size $m=n$.
In this subsection, we characterize all connected unicyclic graphs satisfying $EC(G)=m$.\\

Let $\Theta$ denote the family of unicyclic graphs defined as follows.

(a) The cycle $C_n$, where $3\le n\le6$.

(b) Graphs $G$ obtained from $C_3$ by attaching one or more stars, where the center of each star is identified with a vertex of $C_3$, such as the graphs $G_{5}, G_{6}$, and $G_{7}$ shown in Figure~\ref{f2}.\\

(c) Graphs $G$ obtained from $C_3$ and a double star $S_{p,q}$
by identifying one vertex of $C_3$ with one of the centers of the double star, such as graph $G_{8}$ in Figure~\ref{f2}.\\

(d) Graphs $G$ obtained from $C_4$ by attaching one or two stars, where the center of each star is identified with a vertex of $C_4$, such as the graphs $G_{9}, G_{10}$, and $G_{11}$ shown in Figure~\ref{f2}.\\

(e) Graphs $G$ obtained from $C_5$ by attaching a star, where one vertex of $C_5$ is identified with the center of the star, such as graph $G_{12}$ in Figure~\ref{f2}.

\begin{figure}[!t]
\begin{tikzpicture}[node distance=0.85cm, inner sep=0pt]
\node[state, minimum size=1mm] (C)[fill=black] {$$};
 \node[state, minimum size=1mm] (D)[below left of=C, fill=black] {$$};
  \node[state, minimum size=1mm] (E)[below right of=C, fill=black] {$$};
  \node(L)[below right of=D] { $\small{G_{1}}$};
  \path (C) edge (D);
 \path (C) edge (E);
 \path (D) edge (E);
 \node(L)[below right of=D] { $\small{G_{1}}$};
  \node[state, minimum size=1mm] (F)[right of=C, fill=black] {$$};
 \node[state, minimum size=1mm] (G)[right of=F, fill=black] {$$};
  \node[state, minimum size=1mm] (H)[below of=F, fill=black] {$$};
   \node[state, minimum size=1mm] (I)[below of=G, fill=black] {$$};
  \path (F) edge (G);
 \path (H) edge (I);
 \path (F) edge (H);
 \path (G) edge (I);
  \node(L_1)[below right of=H] { $\small{G_{2}}$};
  \node[state, minimum size=1mm] (I_1)[right of=G, fill=black] {$$};
 \node[state, minimum size=1mm] (I_2)[right of=I_1, fill=black] {$$};
  \node[state, minimum size=1mm] (I_3)[below of=I_1, fill=black] {$$};
   \node[state, minimum size=1mm] (I_4)[below of=I_2, fill=black] {$$};
  \node[state, minimum size=1mm] (I_5)[below left of=I_4, fill=black] {$$};
   \node(L_2)[below of=I_4] { $\small{G_{3}}$};
  \path (I_1) edge (I_2);
 \path (I_1) edge (I_3);
 \path (I_2) edge (I_4);
 \path (I_3) edge (I_5);
 \path (I_4) edge (I_5);
  \node[state, minimum size=1mm] (I_6)[right of=I_2, fill=black] {$$};
 \node[state, minimum size=1mm] (I_7)[right of=I_6, fill=black] {$$};
  \node[state, minimum size=1mm] (I_8)[below of=I_6, fill=black] {$$};
   \node[state, minimum size=1mm] (I_9)[right of=I_8, fill=black] {$$};
  \node[state, minimum size=1mm] (I_10)[below of=I_8, fill=black] {$$};
  \node[state, minimum size=1mm] (I_11)[below of=I_9, fill=black] {$$};
   \node(L_3)[below right of=I_10] { $\small{G_{4}}$};
  \path (I_6) edge (I_7);
 \path (I_10) edge (I_11);
 \path (I_6) edge (I_8);
 \path (I_7) edge (I_9);
 \path (I_8) edge (I_10);
 \path (I_9) edge (I_11);
 \node[state, minimum size=1mm] (I_12)[right of=I_7, fill=black] {$$};
 \node[state, minimum size=1mm] (I_13)[below of=I_12, fill=black] {$$};
  \node[state, minimum size=1mm] (I_14)[right of=I_13, fill=black] {$$};
   \node[state, minimum size=1mm] (I_15)[right of=I_14, fill=black] {$$};
 \node(L_4)[below right of=I_14] { $\small{G_{5}}$};
  \path (I_12) edge (I_13);
 \path (I_12) edge (I_14);
 \path (I_13) edge (I_14);
 \path (I_14) edge (I_15);

 \node[state, minimum size=1mm] (I_16)[right of=I_15, fill=black] {$$};
 \node[state, minimum size=1mm] (I_17)[below of=I_16, fill=black] {$$};
  \node[state, minimum size=1mm] (I_18)[right of=I_16, fill=black] {$$};
   \node[state, minimum size=1mm] (I_19)[above of=I_16, fill=black] {$$};
   \node[state, minimum size=1mm] (I_20)[right of=I_18, fill=black] {$$};
 \node(L_5)[below right of=I_18] { $\small{G_{6}}$};
  \path (I_16) edge (I_17);
 \path (I_16) edge (I_18);
 \path (I_17) edge (I_18);
 \path (I_16) edge (I_19);
 \path (I_16) edge (I_20);
 \end{tikzpicture}
 \begin{tikzpicture}[node distance=0.85cm, inner sep=0pt]
 \node[state, minimum size=1mm] (I_21)[below of=I, fill=black] {$$};
 \node[state, minimum size=1mm] (I_22)[below of=I_21, fill=black] {$$};
  \node[state, minimum size=1mm] (I_23)[right of=I_22, fill=black] {$$};
   \node[state, minimum size=1mm] (I_24)[above of=I_21, fill=black] {$$};
   \node[state, minimum size=1mm] (I_25)[left of=I_22, fill=black] {$$};
    \node[state, minimum size=1mm] (I_26)[right of=I_23, fill=black] {$$};
 \node(L)[below right of=I_23] { $\small{G_{7}}$};
  \path (I_21) edge (I_22);
 \path (I_21) edge (I_23);
 \path (I_22) edge (I_23);
 \path (I_21) edge (I_24);
 \path (I_22) edge (I_25);
 \path (I_23) edge (I_26);
\end{tikzpicture}
\begin{tikzpicture}[node distance=0.85cm, inner sep=0pt]
 \node[state, minimum size=1mm] (I_27)[below of=I_5, fill=black] {$$};
 \node[state, minimum size=1mm] (I_28)[below of=I_27, fill=black] {$$};
  \node[state, minimum size=1mm] (I_29)[right of=I_28, fill=black] {$$};
    \node[state, minimum size=1mm] (I_30)[right of=I_29, fill=black] {$$};
     \node[state, minimum size=1mm] (I_31)[above of=I_30, fill=black] {$$};
 \node(L)[below right of=I_29] { $\small{G_{8}}$};
  \path (I_27) edge (I_28);
 \path (I_27) edge (I_29);
 \path (I_28) edge (I_29);
 \path (I_29) edge (I_30);
 \path (I_30) edge (I_31);
 \end{tikzpicture}
 \begin{tikzpicture}[node distance=0.85cm, inner sep=0pt]
  \node[state, minimum size=1mm] (I_32)[right of=I_31, fill=black] {$$};
 \node[state, minimum size=1mm] (I_33)[right of=I_32, fill=black] {$$};
  \node[state, minimum size=1mm] (I_34)[below of=I_32, fill=black] {$$};
   \node[state, minimum size=1mm] (I_35)[right of=I_34, fill=black] {$$};
   \node[state, minimum size=1mm] (I_36)[right of=I_35, fill=black] {$$};
    \node(L)[below right of=I_34] { $\small{G_{9}}$};
  \path (I_32) edge (I_33);
 \path (I_34) edge (I_35);
 \path (I_32) edge (I_34);
 \path (I_33) edge (I_35);
 \path (I_35) edge (I_36);
 \end{tikzpicture}
\begin{tikzpicture}[node distance=0.85cm, inner sep=0pt]
 \node[state, minimum size=1mm] (I_37)[below of=I_13, fill=black] {$$};
 \node[state, minimum size=1mm] (I_38)[right of=I_37, fill=black] {$$};
  \node[state, minimum size=1mm] (I_39)[below of=I_37, fill=black] {$$};
   \node[state, minimum size=1mm] (I_40)[right of=I_39, fill=black] {$$};
   \node[state, minimum size=1mm] (I_41)[right of=I_38, fill=black] {$$};
     \node[state, minimum size=1mm] (I_42)[right of=I_40, fill=black] {$$};
      \node(L)[below right of=I_39] { $\small{G_{10}}$};
  \path (I_37) edge (I_38);
 \path (I_39) edge (I_40);
 \path (I_37) edge (I_39);
 \path (I_38) edge (I_40);
 \path (I_38) edge (I_41);
 \path (I_40) edge (I_42);
 \end{tikzpicture}
 \begin{tikzpicture}[node distance=0.85cm, inner sep=0pt]
 \node[state, minimum size=1mm] (I_43)[below of=I_14, fill=black] {$$};
 \node[state, minimum size=1mm] (I_44)[right of=I_43, fill=black] {$$};
  \node[state, minimum size=1mm] (I_45)[below of=I_43, fill=black] {$$};
   \node[state, minimum size=1mm] (I_46)[right of=I_45, fill=black] {$$};
   \node[state, minimum size=1mm] (I_47)[left of=I_43, fill=black] {$$};
     \node[state, minimum size=1mm] (I_48)[right of=I_46, fill=black] {$$};
      \node(L)[below right of=I_45] { $\small{G_{11}}$};
  \path (I_43) edge (I_44);
 \path (I_45) edge (I_46);
 \path (I_43) edge (I_45);
 \path (I_44) edge (I_46);
 \path (I_43) edge (I_47);
 \path (I_46) edge (I_48);
 \end{tikzpicture}
 \begin{tikzpicture}[node distance=0.85cm, inner sep=0pt]
 \node[state, minimum size=1mm] (I_49)[below of=I_28, fill=black] {$$};
 \node[state, minimum size=1mm] (I_50)[right of=I_49, fill=black] {$$};
  \node[state, minimum size=1mm] (I_51)[below of=I_49, fill=black] {$$};
   \node[state, minimum size=1mm] (I_52)[right of=I_51, fill=black] {$$};
   \node[state, minimum size=1mm] (I_53)[below right of=I_51, fill=black] {$$};
     \node[state, minimum size=1mm] (I_54)[right of=I_52, fill=black] {$$};
      \node(L)[below right of=I_52] { $\small{G_{12}}$};
  \path (I_49) edge (I_50);
 \path (I_49) edge (I_51);
 \path (I_50) edge (I_52);
 \path (I_51) edge (I_53);
 \path (I_52) edge (I_53);
 \path (I_52) edge (I_54);
 \end{tikzpicture} 
 \begin{center}
 \caption { Graphs belonging to $\Theta $}
 \label{f2}
\end{center}
\end{figure}
\begin{theorem}
\label{the-char-n}
For any connected unicyclic graph $G$ of order $n$ (and hence size $m=n$), we have
$EC(G)=m$ if and only if $G\in\Theta$.
\end{theorem}
\begin{proof}
If $G\in \Theta$, it is straightforward to verify that we have $EC(G)=m$. Conversely, assume that $EC(G)=n=m$.
Let $l$ be the length of the longest path $P$ in $G$, and let $P=(v_1, v_2, \ldots, v_{l+1})$.
For $EC(G)=m$, the singleton edge partition must be an $ec$-partition. We first show that  $l\leq 5$.\\

Suppose $l\geq 6$. Then we have $P=v_1v_2v_3v_4v_5v_6v_7\cdots$. 
Since $G$ is unicyclic, exactly one cycle is present. We distinguish the following cases according to the position of this cycle relative to the longest path $P$.

(i) If the vertices $v_1$ and $v_k$, where $k\ge 7$, are adjacent, then $G$ contains a cycle $C_k$.
By Observation~\ref{obs11}, we have $EC(C_k)\leq 6$.
Now, if $G=C_k$, it is obvious that $EC(G)\neq n$.
If $n\ge k+1$, then $G$ contains at least one pendant edge. A direct inspection shows that such a pendant edge cannot form an edge coalition with any other edge of $G$. 

(ii) If the vertices $v_1$ and $v_6$ are adjacent, then $G$ consists of a single cycle $C_6$ and at least one pendant edge $e$.
A direct inspection shows that the pendant edge $e$ cannot form an edge coalition with any other edge of $G$.

(iii) If the vertices $v_1$ and $v_5$ are adjacent, then $G$ consists of a single cycle $C_5$ and at least one hanging path $P_3$ or at least two pendant edges. In this case, at least one of the pendant edges or one of the edges of the cycle $C_5$ has no edge coalition in $G$.
Thus, $EC(G)\neq n$. 
An analogous argument shows that if the vertex   $v_1$ is adjacent to $v_4$ or $v_3$, respectively, and $l\ge 6$, then at least one of the pendant edges has no edge coalition in $G$. Therefore, every unicyclic graph with $EC(G)=m$ must satisfy $l\le5$.
We now distinguish the remaining possibilities according to the value of $l$.

It remains to show that if a unicyclic graph $G$ with $l\le5$ does not belong to $\Theta$, then $EC(G)\neq m$.

Suppose that $l=5$. Then it must be one of the following graphs in Figure~\ref{f3}, 
where any pendant edge $e$ incident to a vertex $v$ in the cycle may be replaced by a star with center $v$. 
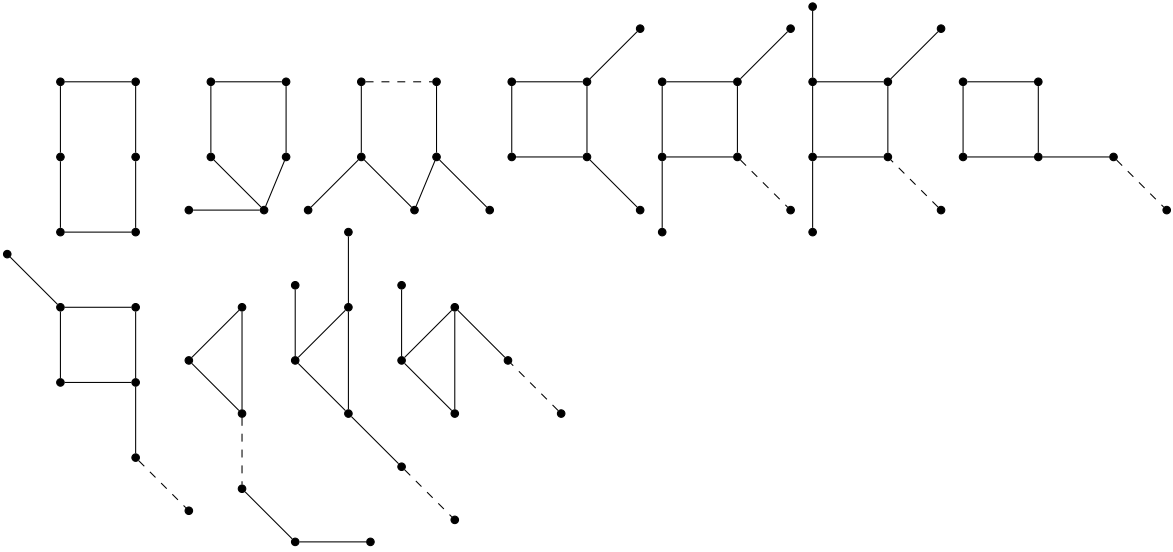
\begin{figure}[!htp]
 \begin{tikzpicture}[scale=0.9,node distance=0.9cm, inner sep=0pt]
  \node[state, minimum size=1mm] (I_6)[right of=I_2, fill=black] {$$};
 \node[state, minimum size=1mm] (I_7)[right of=I_6, fill=black] {$$};
  \node[state, minimum size=1mm] (I_8)[below of=I_6, fill=black] {$$};
   \node[state, minimum size=1mm] (I_9)[right of=I_8, fill=black] {$$};
  \node[state, minimum size=1mm] (I_10)[below of=I_8, fill=black] {$$};
  \node[state, minimum size=1mm] (I_11)[below of=I_9, fill=black] {$$};
  \path (I_6) edge (I_7);
 \path (I_10) edge (I_11);
 \path (I_6) edge (I_8);
 \path (I_7) edge (I_9);
 \path (I_8) edge (I_10);
 \path (I_9) edge (I_11);
  \node[state, minimum size=1mm] (I_12)[right of=I_7, fill=black] {$$};
 \node[state, minimum size=1mm] (I_13)[right of=I_12, fill=black] {$$};
  \node[state, minimum size=1mm] (I_14)[below of=I_12, fill=black] {$$};
   \node[state, minimum size=1mm] (I_15)[below of=I_13, fill=black] {$$};
  \node[state, minimum size=1mm] (I_16)[below right of=I_14, fill=black] {$$};
  \node[state, minimum size=1mm] (I_17)[left of=I_16, fill=black] {$$};
  \path (I_12) edge (I_13);
 \path (I_12) edge (I_14);
 \path (I_13) edge (I_15);
 \path (I_14) edge (I_16);
 \path (I_15) edge (I_16);
 \path (I_16) edge (I_17);
  \node[state, minimum size=1mm] (I_18)[right of=I_13, fill=black] {$$};
 \node[state, minimum size=1mm] (I_19)[right of=I_18, fill=black] {$$};
 \node[state, minimum size=1mm] (I_20)[below  of=I_18, fill=black] {$$};
  \node[state, minimum size=1mm] (I_21)[below of=I_19, fill=black] {$$};
  \node[state, minimum size=1mm] (I_22)[below right of=I_20, fill=black] {$$};
  \node[state, minimum size=1mm] (I_23)[below right of=I_21, fill=black] {$$};
   \node[state, minimum size=1mm] (I_24)[below left of=I_20, fill=black] {$$};
\path[dashed](I_18)edge(I_19);
 \path (I_18) edge (I_20);
 \path (I_19) edge (I_21);
 \path (I_20) edge (I_22);
 \path (I_21) edge (I_22);
 \path (I_20) edge (I_24);
 \path (I_21) edge (I_23);
  \node[state, minimum size=1mm] (I_25)[right of=I_19, fill=black] {$$};
 \node[state, minimum size=1mm] (I_26)[right of=I_25, fill=black] {$$};
  \node[state, minimum size=1mm] (I_27)[below of=I_25, fill=black] {$$};
   \node[state, minimum size=1mm] (I_28)[below of=I_26, fill=black] {$$};
  \node[state, minimum size=1mm] (I_29)[above right of=I_26, fill=black] {$$};
  \node[state, minimum size=1mm] (I_30)[below right of=I_28, fill=black] {$$};
  \path (I_25) edge (I_26);
 \path (I_27) edge (I_28);
 \path (I_25) edge (I_27);
 \path (I_26) edge (I_28);
 \path (I_29) edge (I_26);
 \path (I_30) edge (I_28);
  \node[state, minimum size=1mm] (I_31)[right of=I_26, fill=black] {$$};
 \node[state, minimum size=1mm] (I_32)[right of=I_31, fill=black] {$$};
  \node[state, minimum size=1mm] (I_33)[below of=I_31, fill=black] {$$};
   \node[state, minimum size=1mm] (I_34)[below of=I_32, fill=black] {$$};
  \node[state, minimum size=1mm] (I_35)[above right of=I_32, fill=black] {$$};
  \node[state, minimum size=1mm] (I_36)[below right of=I_34, fill=black] {$$};
   \node[state, minimum size=1mm] (I_37)[below of=I_33, fill=black] {$$};

  \path (I_31) edge (I_32);
 \path (I_31) edge (I_33);
 \path (I_32) edge (I_34);
 \path (I_33) edge (I_34);
 \path (I_32) edge (I_35);
 \path[dashed] (I_34) edge (I_36);
 \path (I_37) edge (I_33);

  \node[state, minimum size=1mm] (I_38)[right of=I_32, fill=black] {$$};
 \node[state, minimum size=1mm] (I_39)[right of=I_38, fill=black] {$$};
  \node[state, minimum size=1mm] (I_40)[below of=I_38, fill=black] {$$};
   \node[state, minimum size=1mm] (I_41)[below of=I_39, fill=black] {$$};
  \node[state, minimum size=1mm] (I_42)[above right of=I_39, fill=black] {$$};
  \node[state, minimum size=1mm] (I_43)[below right of=I_41, fill=black] {$$};
   \node[state, minimum size=1mm] (I_44)[below of=I_40, fill=black] {$$};
\node[state, minimum size=1mm] (I_45)[above of=I_38, fill=black] {$$};
  \path (I_38) edge (I_39);
 \path (I_38) edge (I_40);
 \path (I_40) edge (I_41);
 \path (I_39) edge (I_41);
 \path (I_42) edge (I_39);
 \path[dashed] (I_43) edge (I_41);
 \path (I_45) edge (I_38);
 \path (I_44) edge (I_40);

  \node[state, minimum size=1mm] (I_46)[right of=I_39, fill=black] {$$};
 \node[state, minimum size=1mm] (I_47)[right of=I_46, fill=black] {$$};
  \node[state, minimum size=1mm] (I_48)[below of=I_46, fill=black] {$$};
   \node[state, minimum size=1mm] (I_49)[below of=I_47, fill=black] {$$};
  \node[state, minimum size=1mm] (I_50)[right of=I_49, fill=black] {$$};
 \node[state, minimum size=1mm] (I_51)[below right of=I_50, fill=black] {$$};
  \path (I_46) edge (I_47);
 \path (I_48) edge (I_49);
 \path (I_46) edge (I_48);
 \path (I_47) edge (I_49);
 \path (I_49) edge (I_50);
 \path[dashed] (I_50) edge (I_51);

  \node[state, minimum size=1mm] (I_52)[below of=I_10, fill=black] {$$};
 \node[state, minimum size=1mm] (I_53)[right of=I_52, fill=black] {$$};
  \node[state, minimum size=1mm] (I_54)[below of=I_52, fill=black] {$$};
   \node[state, minimum size=1mm] (I_55)[below of=I_53, fill=black] {$$};
  \node[state, minimum size=1mm] (I_56)[below of=I_55, fill=black] {$$};
  \node[state, minimum size=1mm] (I_57)[below right of=I_56, fill=black] {$$};
\node[state, minimum size=1mm] (I_58)[above left of=I_52, fill=black] {$$};
  \path (I_52) edge (I_53);
 \path (I_54) edge (I_55);
 \path (I_52) edge (I_54);
 \path (I_53) edge (I_55);
 \path (I_55) edge (I_56);
 \path[dashed] (I_56) edge (I_57);
  \path (I_52) edge (I_58);

   \node[state, minimum size=1mm] (I_59)[below right of=I_53, fill=black] {$$};
 \node[state, minimum size=1mm] (I_60)[above right of=I_59, fill=black] {$$};
  \node[state, minimum size=1mm] (I_61)[below right of=I_59, fill=black] {$$};
   \node[state, minimum size=1mm] (I_62)[below of=I_61, fill=black] {$$};
  \node[state, minimum size=1mm] (I_63)[below right of=I_62, fill=black] {$$};
  \node[state, minimum size=1mm] (I_64)[right of=I_63, fill=black] {$$};

  \path (I_59) edge (I_60);
 \path (I_59) edge (I_61);
 \path (I_60) edge (I_61);
 \path[dashed] (I_61) edge (I_62);
 \path (I_62) edge (I_63);
 \path (I_63) edge (I_64);

  \node[state, minimum size=1mm] (I_65)[below right of=I_60, fill=black] {$$};
 \node[state, minimum size=1mm] (I_66)[above right of=I_65, fill=black] {$$};
  \node[state, minimum size=1mm] (I_67)[below right of=I_65, fill=black] {$$};
   \node[state, minimum size=1mm] (I_68)[above of=I_65, fill=black] {$$};
  \node[state, minimum size=1mm] (I_69)[above of=I_66, fill=black] {$$};
  \node[state, minimum size=1mm] (I_70)[below right of=I_67, fill=black] {$$};
  \node[state, minimum size=1mm] (I_71)[below right of=I_70, fill=black] {$$};

  \path (I_65) edge (I_66);
 \path (I_65) edge (I_67);
 \path (I_66) edge (I_67);
 \path (I_68) edge (I_65);
 \path (I_70) edge (I_67);
 \path[dashed] (I_70) edge (I_71);
 \path (I_69) edge (I_66);

  \node[state, minimum size=1mm] (I_72)[below right of=I_66, fill=black] {$$};
 \node[state, minimum size=1mm] (I_73)[above right of=I_72, fill=black] {$$};
  \node[state, minimum size=1mm] (I_74)[below right of=I_72, fill=black] {$$};
   \node[state, minimum size=1mm] (I_75)[above of=I_72, fill=black] {$$};
   \node[state, minimum size=1mm] (I_76)[below right of=I_73, fill=black] {$$};
  \node[state, minimum size=1mm] (I_77)[below right of=I_76, fill=black] {$$};
  \path (I_72) edge (I_73);
 \path (I_72) edge (I_74);
 \path (I_73) edge (I_74);
 \path (I_72) edge (I_75);
 \path (I_76) edge (I_73);
 \path[dashed] (I_77) edge (I_76);
 \end{tikzpicture}
\caption{ All unicyclic graphs in which the maximum path length is $5$.}
\label{f3}
\end{figure}
These are precisely the unicyclic graphs whose longest path has length $5$.
In each of the graphs that is not in $\Theta$, we illustrate by a dashed edge an edge that does not form a coalition with any other edge, implying that for these graphs, $EC(G)\neq n$.

If $l=4$, then $G$ is one of the graphs shown in Figure~\ref{f4},
where any pendant edge $e$ incident to a vertex $v$  may be replaced by a star centered at $v$.
 These graphs are all unicyclic with the  maximum path length being $4$. 
Hence every unicyclic graph with longest path of length four belongs to the family $\Theta$.
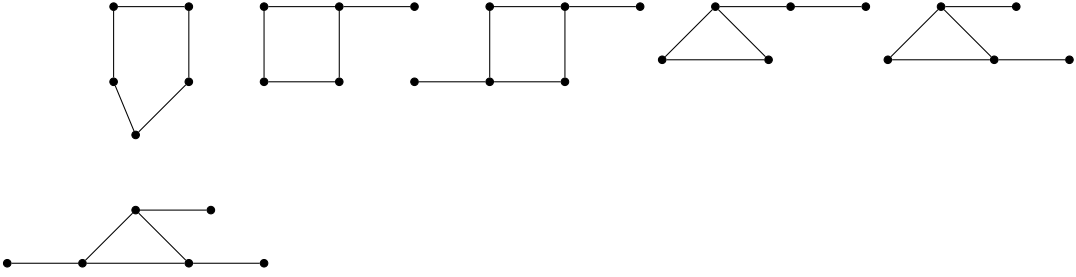
\begin{figure}[!htp]
  \begin{tikzpicture}[node distance=0.9cm, inner sep=0pt]
  \node[state, minimum size=1mm] (I_1)[fill=black] {$$};
 \node[state, minimum size=1mm] (I_2)[right of=I_1, fill=black] {$$};
  \node[state, minimum size=1mm] (I_3)[below of=I_1, fill=black] {$$};
   \node[state, minimum size=1mm] (I_4)[below of=I_2, fill=black] {$$};
  \node[state, minimum size=1mm] (I_5)[below left of=I_4, fill=black] {$$};
  \path (I_1) edge (I_2);
 \path (I_3) edge (I_1);
 \path (I_4) edge (I_2);
 \path (I_3) edge (I_5);
 \path (I_4) edge (I_5);

  \node[state, minimum size=1mm] (I_6)[right of=I_2,fill=black] {$$};
 \node[state, minimum size=1mm] (I_7)[right of=I_6, fill=black] {$$};
  \node[state, minimum size=1mm] (I_8)[below of=I_6, fill=black] {$$};
   \node[state, minimum size=1mm] (I_9)[below of=I_7, fill=black] {$$};
  \node[state, minimum size=1mm] (I_10)[right of=I_7, fill=black] {$$};

  \path (I_6) edge (I_7);
 \path (I_8) edge (I_9);
 \path (I_6) edge (I_8);
 \path (I_7) edge (I_9);
 \path (I_10) edge (I_7);

  \node[state, minimum size=1mm] (I_11)[right of=I_10,fill=black] {$$};
 \node[state, minimum size=1mm] (I_12)[right of=I_11, fill=black] {$$};
  \node[state, minimum size=1mm] (I_13)[below of=I_11, fill=black] {$$};
   \node[state, minimum size=1mm] (I_14)[below of=I_12, fill=black] {$$};
  \node[state, minimum size=1mm] (I_15)[right of=I_12, fill=black] {$$};
  \node[state, minimum size=1mm] (I_16)[left of=I_13, fill=black] {$$};
  \path (I_11) edge (I_12);
 \path (I_13) edge (I_14);
 \path (I_11) edge (I_13);
 \path (I_12) edge (I_14);
 \path (I_15) edge (I_12);
\path (I_16) edge (I_13);

 \node[state, minimum size=1mm] (I_17)[right of=I_15,fill=black] {$$};
 \node[state, minimum size=1mm] (I_18)[below left of=I_17, fill=black] {$$};
  \node[state, minimum size=1mm] (I_19)[below right of=I_17, fill=black] {$$};
   \node[state, minimum size=1mm] (I_20)[right of=I_17, fill=black] {$$};
  \node[state, minimum size=1mm] (I_21)[right of=I_20, fill=black] {$$};

  \path (I_17) edge (I_18);
 \path (I_17) edge (I_19);
 \path (I_18) edge (I_19);
 \path (I_17) edge (I_20);
 \path (I_20) edge (I_21);

 \node[state, minimum size=1mm] (I_22)[right of=I_21,fill=black] {$$};
 \node[state, minimum size=1mm] (I_23)[below left of=I_22, fill=black] {$$};
  \node[state, minimum size=1mm] (I_24)[below right of=I_22, fill=black] {$$};
   \node[state, minimum size=1mm] (I_25)[right of=I_24, fill=black] {$$};
  \node[state, minimum size=1mm] (I_26)[right of=I_22, fill=black] {$$};

  \path (I_22) edge (I_23);
 \path (I_22) edge (I_24);
 \path (I_23) edge (I_24);
 \path (I_25) edge (I_24);
 \path (I_26) edge (I_22);

  \node[state, minimum size=1mm] (I_27)[below of=I_5,fill=black] {$$};
 \node[state, minimum size=1mm] (I_28)[below left of=I_27, fill=black] {$$};
  \node[state, minimum size=1mm] (I_29)[below right of=I_27, fill=black] {$$};
   \node[state, minimum size=1mm] (I_30)[right of=I_27, fill=black] {$$};
  \node[state, minimum size=1mm] (I_31)[right of=I_29, fill=black] {$$};
   \node[state, minimum size=1mm] (I_32)[left of=I_28, fill=black] {$$};

  \path (I_27) edge (I_28);
 \path (I_27) edge (I_29);
 \path (I_28) edge (I_29);
 \path (I_30) edge (I_27);
 \path (I_31) edge (I_29);
 \path (I_32) edge (I_28);
 \end{tikzpicture}
\caption{All unicyclic graphs in which the maximum path length is  $4$.}
\label{f4}
\end{figure}
If $l=3$, then $G$ is either $C_4$ or is obtained from  $C_3$ by identifying one vertex of $C_3$ with the center of a star.
If $l=2$, then $G$ is $C_3$. 
Therefore, every connected unicyclic graph satisfying $EC(G)=m$ belongs to the family $\Theta$. This completes the proof.
\end{proof}
%
\subsection{Trees}
Let $T$ be a tree of order $n$. Then $|E(T)|=m=n-1$.
\begin{theorem}
\label{the-t_n-1}
Let $T$ be a tree of order $n$ and size $m$. Then $EC(T)=m=n-1$   if and only if $T\in\Phi$, 
where $\Phi$ denotes the family consisting of all trees of diameter at most $3$, together with all trees of diameter $4$ whose central vertex on every longest path has degree $2$, as illustrated in Figure~\ref{f5}.
\end{theorem}
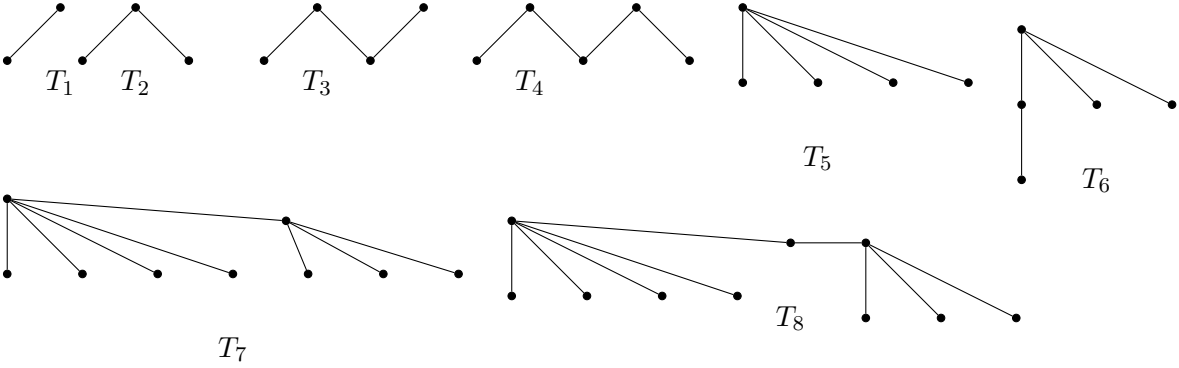
\begin{figure}[!htp]
 \begin{tikzpicture}[node distance=0.9cm, inner sep=0pt]

  \node[state, minimum size=1mm] (J_1)[fill=black] {$$};
 \node[state, minimum size=1mm] (J_2)[below left of=J_1, fill=black] {$$};
 \node(L)[below of=J_1] { $\small{T_{1}}$};
  \path (J_1) edge (J_2);

  \node[state, minimum size=1mm] (J_2)[right of=J_1, fill=black] {$$};
 \node[state, minimum size=1mm] (J_3)[below left of=J_2, fill=black] {$$};
 \node[state, minimum size=1mm] (J_4)[below right of=J_2, fill=black] {$$};
  \node(L_1)[below of=J_2] { $\small{T_{2}}$};
  \path (J_2) edge (J_3);
  \path (J_2) edge (J_4);
   \node[state, minimum size=1mm] (J_5)[right of=J_4, fill=black] {$$};
 \node[state, minimum size=1mm] (J_6)[above right of=J_5, fill=black] {$$};
 \node[state, minimum size=1mm] (J_7)[below right of=J_6, fill=black] {$$};
  \node[state, minimum size=1mm] (J_8)[above right of=J_7, fill=black] {$$};
  \node(L_2)[below of=J_6] { $\small{T_{3}}$};
  \path (J_5) edge (J_6);
  \path (J_6) edge (J_7);
  \path (J_7) edge (J_8);
  \node[state, minimum size=1mm] (J_9)[below right of=J_8, fill=black] {$$};
 \node[state, minimum size=1mm] (J_10)[above right of=J_9, fill=black] {$$};
 \node[state, minimum size=1mm] (J_11)[below right of=J_10, fill=black] {$$};
  \node[state, minimum size=1mm] (J_12)[above right of=J_11, fill=black] {$$};
  \node[state, minimum size=1mm] (J_13)[below right of=J_12, fill=black] {$$};
    \node(L_3)[below of=J_10] { $\small{T_{4}}$};
  \path (J_9) edge (J_10);
  \path (J_10) edge (J_11);
  \path (J_11) edge (J_12);
 \path (J_12) edge (J_13);
 \node[state, minimum size=1mm] (J_14)[above right of=J_13, fill=black] {$$};
 \node[state, minimum size=1mm] (J_15)[below of=J_14, fill=black] {$$};
 \node[state, minimum size=1mm] (J_16)[right of=J_15, fill=black] {$$};
  \node[state, minimum size=1mm] (J_17)[right of=J_16, fill=black] {$$};
  \node[state, minimum size=1mm] (J_18)[right of=J_17, fill=black] {$$};
    \node(L_4)[below of=J_16] { $\small{T_{5}}$};
  \path (J_14) edge (J_15);
  \path (J_14) edge (J_16);
  \path (J_14) edge (J_17);
 \path (J_14) edge (J_18);

   \node[state, minimum size=1mm] (J_38)[above right of=J_18, fill=black] {$$};
 \node[state, minimum size=1mm] (J_39)[below of=J_38, fill=black] {$$};
 \node[state, minimum size=1mm] (J_40)[right of=J_39, fill=black] {$$};
  \node[state, minimum size=1mm] (J_41)[right of=J_40, fill=black] {$$};
  \node[state, minimum size=1mm] (J_42)[below of=J_39, fill=black] {$$};
  \node(L_5)[below of=J_40] { $\small{T_{6}}$};
  \path (J_38) edge (J_39);
  \path (J_38) edge (J_40);
  \path (J_38) edge (J_41);
 \path (J_39) edge (J_42);
 \end{tikzpicture}

  \begin{tikzpicture}[node distance=0.9cm, inner sep=0pt]

  \node[state, minimum size=1mm] (J_19)[fill=black] {$$};
 \node[state, minimum size=1mm] (J_20)[below of=J_19, fill=black] {$$};
 \node[state, minimum size=1mm] (J_21)[right of=J_20, fill=black] {$$};
  \node[state, minimum size=1mm] (J_22)[right of=J_21, fill=black] {$$};
  \node[state, minimum size=1mm] (J_23)[right of=J_22, fill=black] {$$};
   \node[state, minimum size=1mm] (J_24)[above right of=J_23, fill=black] {$$};
  \node[state, minimum size=1mm] (J_25)[right of=J_23, fill=black] {$$};
  \node[state, minimum size=1mm] (J_26)[right of=J_25, fill=black] {$$};
   \node[state, minimum size=1mm] (J_27)[right of=J_26, fill=black] {$$};
  \node(L)[below of=J_23] { $\small{T_{7}}$};
  \path (J_19) edge (J_20);
  \path (J_19) edge (J_21);
  \path (J_19) edge (J_22);
 \path (J_19) edge (J_23);
   \path (J_24) edge (J_19);
  \path (J_24) edge (J_25);
  \path (J_24) edge (J_26);
 \path (J_24) edge (J_27);

  \node[state, minimum size=1mm] (J_28)[above right of=J_27, fill=black] {$$};
 \node[state, minimum size=1mm] (J_29)[below of=J_28, fill=black] {$$};
 \node[state, minimum size=1mm] (J_30)[right of=J_29, fill=black] {$$};
  \node[state, minimum size=1mm] (J_31)[right of=J_30, fill=black] {$$};
  \node[state, minimum size=1mm] (J_32)[right of=J_31, fill=black] {$$};
   \node[state, minimum size=1mm] (J_33)[above right of=J_32, fill=black] {$$};
  \node[state, minimum size=1mm] (J_34)[right of=J_33, fill=black] {$$};
  \node[state, minimum size=1mm] (J_35)[below of=J_34, fill=black] {$$};
   \node[state, minimum size=1mm] (J_36)[right of=J_35, fill=black] {$$};
 \node[state, minimum size=1mm] (J_37)[right of=J_36, fill=black] {$$};
   \node(L_1)[below of=J_33] { $\small{T_{8}}$};
  \path (J_28) edge (J_29);
  \path (J_28) edge (J_30);
  \path (J_28) edge (J_31);
 \path (J_28) edge (J_32);
   \path (J_33) edge (J_28);
  \path (J_33) edge (J_34);
  \path (J_34) edge (J_35);
 \path (J_34) edge (J_36);
 \path (J_34) edge (J_37);
\end{tikzpicture}
\begin{center}
\caption{ Some trees that are in $\Phi$ }  
\label{f5}
\end{center}
\end{figure}
\begin{proof}
In every tree $T$, we have $m=n-1$. Thus, it is clear that $EC(T)\leq m=n-1$.
If $T\in \Phi$, it follows directly that $EC(T)=m$. Conversely, assume that $EC(T)=m$.

Let $l$ be the length of a longest path $P$ in $T$,  where  $P=v_1v_2\ldots v_{l+1}$.
We claim that $l\leq4$.

Suppose $l\geq5$. Then  $P=v_1v_2v_3v_4v_5v_6\cdots v_l$.
So there are at least five edges $e_1,e_2,e_3,e_4$, and $e_5$ in the tree $T$.

Hence, the middle edge $e_3$ cannot form an edge coalition with any other edge of $T$.
Therefore, the singleton edge partition
$X=\{\{e_1\},\{e_2\},\{e_3\},\{e_4\},\cdots,\{e_m\}\}$
is not an edge coalition partition of $T$.
Consequently, $EC(T)\leq m-1$, which leads to a contradiction.
This contradiction proves that $l\le4$.

For $l=4$, there exists a path
$P=v_1v_2v_3v_4v_5$   in $T$ with length $4$.
Suppose that there exists a vertex
$w\in N_T(v_3)\setminus\{v_2,v_4\}$.
Then the edge $e=wv_3$ cannot form an edge coalition with any other edge of $T$.
On the other hand, since the length of a longest path in $T$ is $4$, we have
$N_T(v_1)=\{v_2\}$ and  $N_T(v_5)=\{v_4\}$.
Hence, every tree with longest path of length $4$ satisfying $EC(T)=m$
must satisfy $deg(v_3)=2$, that is, $v_3$ is the central vertex of every longest path,
For $l=3$, the tree $T$ is a double star, such as $T_3$, $T_6$, and $T_7$ in Figure~\ref{f5}.
For $l=2$, the tree $T$ is  $K_{1,n-1}$,
such as  $T_2$ and $T_5$ in Figure~\ref{f5}.
For $l=1$, the tree $T$ is $P_2$.
Therefore, $EC(T)=m$ if and only if $T\in\Phi$, completing the proof.
\end{proof}
%
\subsection{Connected Graphs with Edge Coalition Number Equal to Their Size}
Let $G$ be a connected graph of order $n$ and size $m$.
Throughout this subsection, we characterize the graphs satisfying $EC(G)=m$.
Let $\Psi$ denote the family of connected graphs that are neither trees nor unicyclic graphs, shown in Figure~\ref{f6}.
\begin{figure}[!htp]
 \begin{tikzpicture}[node distance=0.85cm, inner sep=0pt]

\node[state, minimum size=1mm] (F)[fill=black] {$$};
 \node[state, minimum size=1mm] (G)[right of=F, fill=black] {$$};
  \node[state, minimum size=1mm] (H)[below of=F, fill=black] {$$};
   \node[state, minimum size=1mm] (I)[below of=G, fill=black] {$$};
   \node(L_1)[below left of=I] { $\small{H_{1}}$};
  \path (F) edge (G);
 \path (H) edge (I);
 \path (F) edge (H);
 \path (G) edge (I);
 \path (F) edge (I);

\node[state, minimum size=1mm] (J_1)[right of=G, fill=black] {$$};
 \node[state, minimum size=1mm] (J_2)[right of=J_1, fill=black] {$$};
  \node[state, minimum size=1mm] (J_3)[below of=J_1, fill=black] {$$};
   \node[state, minimum size=1mm] (J_4)[below of=J_2, fill=black] {$$};
   \node(L_1)[below left of=J_4] { $\small{H_{2}}$};

  \path (J_1) edge (J_2);
 \path (J_1) edge (J_3);
 \path (J_2) edge (J_4);
 \path (J_3) edge (J_4);
 \path (J_1) edge (J_4);
\path (J_2) edge (J_3);

\node[state, minimum size=1mm] (J_22)[right of=J_2, fill=black] {$$};
 \node[state, minimum size=1mm] (J_23)[right of=J_22, fill=black] {$$};
  \node[state, minimum size=1mm] (J_24)[below of=J_22, fill=black] {$$};
   \node[state, minimum size=1mm] (J_25)[below of=J_23, fill=black] {$$};
 \node[state, minimum size=1mm] (J_26)[above right of=J_23, fill=black] {$$};
    \node(L_1)[below right of=J_24] { $\small{H_{3}}$};
  \path (J_22) edge (J_23);
 \path (J_22) edge (J_24);
 \path (J_23) edge (J_25);
 \path (J_24) edge (J_25);
 \path (J_23) edge (J_26);
 \path (J_22) edge (J_25);

\node[state, minimum size=1mm] (J_26)[right of=J_23, fill=black] {$$};
 \node[state, minimum size=1mm] (J_27)[right of=J_26, fill=black] {$$};
  \node[state, minimum size=1mm] (J_28)[below of=J_26, fill=black] {$$};
   \node[state, minimum size=1mm] (J_29)[below of=J_27, fill=black] {$$};
 \node[state, minimum size=1mm] (J_30)[above right of=J_27, fill=black] {$$};
 \node(L_1)[below left of=J_29] { $\small{H_{4}}$};

  \path (J_26) edge (J_27);
 \path (J_26) edge (J_28);
 \path (J_27) edge (J_29);
 \path (J_28) edge (J_29);
 \path (J_27) edge (J_28);
 \path (J_27) edge (J_30);

\node[state, minimum size=1mm] (J_31)[right of=J_27, fill=black] {$$};
 \node[state, minimum size=1mm] (J_32)[right of=J_31, fill=black] {$$};
  \node[state, minimum size=1mm] (J_33)[below of=J_31, fill=black] {$$};
   \node[state, minimum size=1mm] (J_34)[below of=J_32, fill=black] {$$};
 \node[state, minimum size=1mm] (J_35)[above of=J_31, fill=black] {$$};
 \node[state, minimum size=1mm] (J_36)[below right of=J_34, fill=black] {$$};
 \node(L_1)[below left of=J_34] { $\small{H_{5}}$};
  \path (J_31) edge (J_32);
 \path (J_31) edge (J_33);
 \path (J_33) edge (J_34);
 \path (J_32) edge (J_34);
 \path (J_34) edge (J_36);
 \path (J_31) edge (J_35);
 \path (J_31) edge (J_34);

\node[state, minimum size=1mm] (J_1)[right of=J_32, fill=black] {$$};
 \node[state, minimum size=1mm] (J_2)[right of=J_1, fill=black] {$$};
  \node[state, minimum size=1mm] (J_3)[below of=J_1, fill=black] {$$};
   \node[state, minimum size=1mm] (J_4)[below of=J_2, fill=black] {$$};
 \node[state, minimum size=1mm] (J_5)[above right of=J_2, fill=black] {$$};
 \node[state, minimum size=1mm] (J_6)[below right of=J_4, fill=black] {$$};
 \node(L_1)[below left of=J_4] { $\small{H_{6}}$};
  \path (J_1) edge (J_2);
 \path (J_1) edge (J_3);
 \path (J_2) edge (J_4);
 \path (J_3) edge (J_4);
 \path (J_2) edge (J_5);
 \path (J_4) edge (J_6);
 \path (J_2) edge (J_3);

\node[state, minimum size=1mm] (J_7)[right of=J_2, fill=black] {$$};
 \node[state, minimum size=1mm] (J_8)[right of=J_7, fill=black] {$$};
  \node[state, minimum size=1mm] (J_9)[below of=J_7, fill=black] {$$};
   \node[state, minimum size=1mm] (J_10)[below of=J_8, fill=black] {$$};
  \node[state, minimum size=1mm] (J_11)[above right of=J_8, fill=black] {$$};
   \node(L_1)[below left of=J_10] { $\small{H_{7}}$};
  \path (J_7) edge (J_8);
 \path (J_7) edge (J_9);
 \path (J_9) edge (J_10);
 \path (J_8) edge (J_10);
 \path (J_7) edge (J_10);
\path (J_8) edge (J_9);
\path (J_8) edge (J_11);

\node[state, minimum size=1mm] (J_12)[right of=J_8, fill=black] {$$};
 \node[state, minimum size=1mm] (J_13)[right of=J_12, fill=black] {$$};
  \node[state, minimum size=1mm] (J_14)[below of=J_12, fill=black] {$$};
   \node[state, minimum size=1mm] (J_15)[below of=J_13, fill=black] {$$};
  \node[state, minimum size=1mm] (J_16)[above right of=J_13, fill=black] {$$};
  \node[state, minimum size=1mm] (J_17)[below right of=J_15, fill=black] {$$};
     \node(L_1)[below left of=J_15] { $\small{H_{8}}$};
  \path (J_12) edge (J_13);
 \path (J_12) edge (J_14);
 \path (J_13) edge (J_15);
 \path (J_14) edge (J_15);
 \path (J_12) edge (J_15);
\path (J_13) edge (J_14);
\path (J_13) edge (J_16);
\path (J_15) edge (J_17);

\end{tikzpicture}
\begin{tikzpicture}[node distance=0.85cm, inner sep=0pt]

\node[state, minimum size=1mm] (J_12)[below of=J_14,fill=black] {$$};
 \node[state, minimum size=1mm] (J_13)[right of=J_12, fill=black] {$$};
  \node[state, minimum size=1mm] (J_14)[below of=J_12, fill=black] {$$};
   \node[state, minimum size=1mm] (J_15)[below of=J_13, fill=black] {$$};
  \node[state, minimum size=1mm] (J_16)[above of=J_12, fill=black] {$$};
  \node[state, minimum size=1mm] (J_17)[below right of=J_15, fill=black] {$$};
    \node[state, minimum size=0mm](L_2)[below left of=J_17, fill=white]{$$};
   \node(L_1)[below right of=J_14] { $\small{H_{9}}$};
  \path (J_12) edge (J_13);
 \path (J_12) edge (J_14);
 \path (J_13) edge (J_15);
 \path (J_14) edge (J_15);
 \path (J_12) edge (J_15);
\path (J_13) edge (J_14);
\path (J_12) edge (J_16);
\path (J_15) edge (J_17);

\node[state, minimum size=1mm] (J_18)[right of=J_13, fill=black] {$$};
 \node[state, minimum size=1mm] (J_19)[right of=J_18, fill=black] {$$};
  \node[state, minimum size=1mm] (J_20)[below of=J_18, fill=black] {$$};
   \node[state, minimum size=1mm] (J_21)[below of=J_19, fill=black] {$$};
  \node[state, minimum size=1mm] (J_22)[above right of=J_19, fill=black] {$$};
  \node[state, minimum size=1mm] (J_23)[below right of=J_21, fill=black] {$$};
 \node[state, minimum size=1mm] (J_24)[above of=J_18, fill=black] {$$};
    \node(L_1)[left of=J_23] { $\small{H_{10}}$};
  \path (J_18) edge (J_19);
 \path (J_18) edge (J_20);
 \path (J_20) edge (J_21);
 \path (J_19) edge (J_21);
 \path (J_18) edge (J_21);
\path (J_19) edge (J_20);
\path (J_19) edge (J_22);
\path (J_21) edge (J_23);
\path (J_18) edge (J_24);
\node[state, minimum size=1mm] (J_25)[right of=J_19, fill=black] {$$};
 \node[state, minimum size=1mm] (J_26)[right of=J_25, fill=black] {$$};
  \node[state, minimum size=1mm] (J_27)[below of=J_25, fill=black] {$$};
   \node[state, minimum size=1mm] (J_28)[below of=J_26, fill=black] {$$};
  \node[state, minimum size=1mm] (J_29)[above right of=J_26, fill=black] {$$};
  \node[state, minimum size=1mm] (J_30)[below right of=J_28, fill=black] {$$};
 \node[state, minimum size=1mm] (J_31)[above of=J_25, fill=black] {$$};
    \node(L_1)[left of=J_30] { $\small{H_{11}}$};
  \path (J_25) edge (J_26);
 \path (J_25) edge (J_27);
 \path (J_26) edge (J_28);
 \path (J_27) edge (J_28);
 \path (J_26) edge (J_29);
\path (J_28) edge (J_30);
\path (J_25) edge (J_28);
\path (J_31) edge (J_25);

\node[state, minimum size=1mm] (J_38)[right of=J_26, fill=black] {$$};
 \node[state, minimum size=1mm] (J_39)[right of=J_38, fill=black] {$$};
  \node[state, minimum size=1mm] (J_40)[below of=J_38, fill=black] {$$};
   \node[state, minimum size=1mm] (J_41)[below right of=J_39, fill=black] {$$};
  \node[state, minimum size=1mm] (J_42)[below right of=J_40, fill=black] {$$};
   \node(L_1)[right of=J_42] { $\small{H_{12}}$};
  \path (J_38) edge (J_39);
 \path (J_38) edge (J_40);
 \path (J_40) edge (J_42);
 \path (J_42) edge (J_41);
 \path (J_41) edge (J_39);
\path (J_40) edge (J_41);

\node[state, minimum size=1mm] (J_42)[above right of=J_41, fill=black] {$$};
 \node[state, minimum size=1mm] (J_43)[right of=J_42, fill=black] {$$};
  \node[state, minimum size=1mm] (J_44)[below of=J_42, fill=black] {$$};
   \node[state, minimum size=1mm] (J_45)[below right of=J_43, fill=black] {$$};
  \node[state, minimum size=1mm] (J_46)[below right of=J_44, fill=black] {$$};
  \node(L_1)[right of=J_46] { $\small{H_{13}}$};
  \path (J_42) edge (J_43);
 \path (J_42) edge (J_44);
 \path (J_44) edge (J_46);
 \path (J_46) edge (J_45);
 \path (J_45) edge (J_43);
\path (J_42) edge (J_45);
\path (J_42) edge (J_46);

\node[state, minimum size=1mm] (J_47)[above right of=J_45, fill=black] {$$};
 \node[state, minimum size=1mm] (J_48)[right of=J_47, fill=black] {$$};
  \node[state, minimum size=1mm] (J_49)[below of=J_47, fill=black] {$$};
   \node[state, minimum size=1mm] (J_50)[below right of=J_48, fill=black] {$$};
  \node[state, minimum size=1mm] (J_51)[below right of=J_49, fill=black] {$$};
 \node(L_1)[right of=J_51] { $\small{H_{14}}$};
  \path (J_47) edge (J_48);
 \path (J_47) edge (J_49);
 \path (J_49) edge (J_51);
 \path (J_51) edge (J_50);
 \path (J_50) edge (J_48);
\path (J_48) edge (J_49);
\path (J_47) edge (J_50);

\node[state, minimum size=1mm] (J_52)[above right of=J_50, fill=black] {$$};
 \node[state, minimum size=1mm] (J_53)[right of=J_52, fill=black] {$$};
  \node[state, minimum size=1mm] (J_54)[below of=J_52, fill=black] {$$};
   \node[state, minimum size=1mm] (J_55)[below right of=J_53, fill=black] {$$};
  \node[state, minimum size=1mm] (J_56)[below right of=J_54, fill=black] {$$};
 \node(L_1)[right of=J_56] { $\small{H_{15}}$};
  \path (J_52) edge (J_53);
 \path (J_52) edge (J_54);
 \path (J_54) edge (J_56);
 \path (J_56) edge (J_55);
 \path (J_55) edge (J_54);
\path (J_55) edge (J_52);
\path (J_54) edge (J_53);
\path (J_55) edge (J_53);

\end{tikzpicture}
\begin{tikzpicture}[node distance=0.85cm, inner sep=0pt]

\node[state, minimum size=1mm] (J_57)[fill=black] {$$};
 \node[state, minimum size=1mm] (J_58)[right of=J_57, fill=black] {$$};
  \node[state, minimum size=1mm] (J_59)[below of=J_57, fill=black] {$$};
   \node[state, minimum size=1mm] (J_60)[below right of=J_58, fill=black] {$$};
  \node[state, minimum size=1mm] (J_61)[below right of=J_59, fill=black] {$$};
 \node(L_1)[right of=J_61] { $\small{H_{16}}$};
  \path (J_57) edge (J_58);
 \path (J_57) edge (J_59);
 \path (J_59) edge (J_61);
 \path (J_61) edge (J_60);
 \path (J_60) edge (J_58);
\path (J_58) edge (J_59);
\path (J_58) edge (J_61);
\path (J_57) edge (J_60);

\node[state, minimum size=1mm] (J_62)[above right of=J_60, fill=black] {$$};
 \node[state, minimum size=1mm] (J_63)[right of=J_62, fill=black] {$$};
  \node[state, minimum size=1mm] (J_64)[below of=J_62, fill=black] {$$};
   \node[state, minimum size=1mm] (J_65)[below right of=J_63, fill=black] {$$};
  \node[state, minimum size=1mm] (J_66)[below right of=J_64, fill=black] {$$};
\node(L_1)[right of=J_66] { $\small{H_{17}}$};
  \path (J_62) edge (J_63);
 \path (J_62) edge (J_64);
 \path (J_64) edge (J_66);
 \path (J_66) edge (J_65);
 \path (J_65) edge (J_63);
\path (J_64) edge (J_63);
\path (J_64) edge (J_65);
\path (J_62) edge (J_66);
\path (J_62) edge (J_65);

\node[state, minimum size=1mm] (J_67)[above right of=J_65, fill=black] {$$};
 \node[state, minimum size=1mm] (J_68)[right of=J_67, fill=black] {$$};
  \node[state, minimum size=1mm] (J_69)[below of=J_67, fill=black] {$$};
   \node[state, minimum size=1mm] (J_70)[below right of=J_68, fill=black] {$$};
  \node[state, minimum size=1mm] (J_71)[below right of=J_69, fill=black] {$$};
\node(L_1)[right of=J_71] { $\small{H_{18}}$};
  \path (J_67) edge (J_68);
 \path (J_67) edge (J_69);
 \path (J_69) edge (J_71);
 \path (J_71) edge (J_70);
 \path (J_70) edge (J_68);
\path (J_67) edge (J_70);
\path (J_67) edge (J_71);
\path (J_68) edge (J_71);
\path (J_68) edge (J_69);
 \path (J_70) edge (J_69);
 \path (J_70) edge (J_67);
\node[state, minimum size=1mm] (J_72)[right of=J_70, fill=black] {$$};
 \node[state, minimum size=1mm] (J_73)[right of=J_72, fill=black] {$$};
  \node[state, minimum size=1mm] (J_74)[below left of=J_72, fill=black] {$$};
   \node[state, minimum size=1mm] (J_75)[below right of=J_73, fill=black] {$$};
  \node[state, minimum size=1mm] (J_76)[below right of=J_74, fill=black] {$$};
 \node[state, minimum size=1mm] (J_77)[above right of=J_73, fill=black] {$$};
 \node(L_1)[right of=J_76] { $\small{H_{19}}$};
  \path (J_72) edge (J_73);
 \path (J_72) edge (J_74);
 \path (J_74) edge (J_76);
 \path (J_76) edge (J_75);
 \path (J_75) edge (J_73);
\path (J_73) edge (J_77);
\path (J_74) edge (J_73);

\node[state, minimum size=1mm] (J_78)[above right of=J_75, fill=black] {$$};
 \node[state, minimum size=1mm] (J_79)[right of=J_78, fill=black] {$$};
  \node[state, minimum size=1mm] (J_80)[below of=J_78, fill=black] {$$};
   \node[state, minimum size=1mm] (J_81)[below right of=J_79, fill=black] {$$};
  \node[state, minimum size=1mm] (J_82)[below right of=J_80, fill=black] {$$};
 \node[state, minimum size=1mm] (J_83)[above right of=J_79, fill=black] {$$};
  \node(L_1)[right of=J_82] { $\small{H_{20}}$};
  \path (J_78) edge (J_79);
 \path (J_78) edge (J_80);
 \path (J_80) edge (J_82);
 \path (J_82) edge (J_81);
 \path (J_81) edge (J_79);
\path (J_79) edge (J_83);
\path (J_80) edge (J_81);

\node[state, minimum size=1mm] (J_84)[right of=J_79, fill=black] {$$};
 \node[state, minimum size=1mm] (J_85)[right of=J_84, fill=black] {$$};
  \node[state, minimum size=1mm] (J_86)[below of=J_84, fill=black] {$$};
   \node[state, minimum size=1mm] (J_87)[below right of=J_85, fill=black] {$$};
  \node[state, minimum size=1mm] (J_88)[below right of=J_86, fill=black] {$$};
 \node[state, minimum size=1mm] (J_89)[above right of=J_85, fill=black] {$$};
 \node[state, minimum size=1mm] (J_90)[right of=J_88, fill=black] {$$};
  \node(L_1)[below of=J_88] { $\small{H_{21}}$};
  \path (J_84) edge (J_85);
 \path (J_84) edge (J_86);
 \path (J_86) edge (J_88);
 \path (J_88) edge (J_87);
 \path (J_87) edge (J_85);
\path (J_85) edge (J_89);
\path (J_88) edge (J_90);
\path (J_85) edge (J_88);

\node[state, minimum size=1mm] (J_91)[right of=J_85, fill=black] {$$};
 \node[state, minimum size=1mm] (J_92)[right of=J_91, fill=black] {$$};
  \node[state, minimum size=1mm] (J_93)[below of=J_91, fill=black] {$$};
   \node[state, minimum size=1mm] (J_94)[below right of=J_92, fill=black] {$$};
  \node[state, minimum size=1mm] (J_95)[below right of=J_93, fill=black] {$$};
 \node[state, minimum size=1mm] (J_96)[above right of=J_92, fill=black] {$$};
 \node(L_1)[below of=J_94] { $\small{H_{22}}$};
  \path (J_91) edge (J_92);
 \path (J_91) edge (J_93);
 \path (J_93) edge (J_95);
 \path (J_95) edge (J_94);
 \path (J_92) edge (J_94);
\path (J_92) edge (J_93);
\path (J_92) edge (J_95);
\path (J_92) edge (J_96);

\end{tikzpicture}

\begin{tikzpicture}[node distance=0.85cm, inner sep=0pt]

\node[state, minimum size=1mm] (J_97)[right of=J_92, fill=black] {$$};
 \node[state, minimum size=1mm] (J_98)[right of=J_97, fill=black] {$$};
  \node[state, minimum size=1mm] (J_99)[below of=J_97, fill=black] {$$};
   \node[state, minimum size=1mm] (J_100)[below right of=J_98, fill=black] {$$};
  \node[state, minimum size=1mm] (J_101)[below right of=J_99, fill=black] {$$};
 \node[state, minimum size=1mm] (J_102)[below right of=J_101, fill=black] {$$};
\node(L_1)[right of=J_101] { $\small{H_{23}}$};
  \path (J_97) edge (J_98);
 \path (J_97) edge (J_99);
 \path (J_99) edge (J_101);
 \path (J_101) edge (J_100);
 \path (J_100) edge (J_98);
\path (J_101) edge (J_102);
\path (J_98) edge (J_99);
\path (J_98) edge (J_101);

\node[state, minimum size=1mm] (J_103)[right of=J_98, fill=black] {$$};
 \node[state, minimum size=1mm] (J_104)[right of=J_103, fill=black] {$$};
  \node[state, minimum size=1mm] (J_105)[below of=J_103, fill=black] {$$};
   \node[state, minimum size=1mm] (J_106)[below right of=J_104, fill=black] {$$};
  \node[state, minimum size=1mm] (J_107)[below right of=J_105, fill=black] {$$};
 \node[state, minimum size=1mm] (J_108)[above left of=J_103, fill=black] {$$};
\node(L_1)[below of=J_107] { $\small{H_{24}}$};
  \path (J_103) edge (J_104);
 \path (J_103) edge (J_105);
 \path (J_105) edge (J_107);
 \path (J_107) edge (J_106);
 \path (J_106) edge (J_104);
\path (J_104) edge (J_105);
\path (J_104) edge (J_107);
\path (J_103) edge (J_108);
\node[state, minimum size=1mm] (J_109)[right of=J_104, fill=black] {$$};
 \node[state, minimum size=1mm] (J_110)[right of=J_109, fill=black] {$$};
  \node[state, minimum size=1mm] (J_111)[below of=J_109, fill=black] {$$};
   \node[state, minimum size=1mm] (J_112)[below right of=J_110, fill=black] {$$};
  \node[state, minimum size=1mm] (J_113)[below right of=J_111, fill=black] {$$};
 \node[state, minimum size=1mm] (J_114)[below left of=J_111, fill=black] {$$};
 \node[state, minimum size=1mm] (J_115)[below left of=J_113, fill=black] {$$};
 \node(L_1)[below of=J_112] { $\small{H_{25}}$};
  \path (J_109) edge (J_110);
 \path (J_109) edge (J_111);
 \path (J_111) edge (J_113);
 \path (J_113) edge (J_112);
 \path (J_112) edge (J_110);
\path (J_111) edge (J_114);
\path (J_113) edge (J_115);
\path (J_110) edge (J_111);
\path (J_110) edge (J_113);

\node[state, minimum size=1mm] (J_116)[right of=J_110, fill=black] {$$};
 \node[state, minimum size=1mm] (J_117)[right of=J_116, fill=black] {$$};
  \node[state, minimum size=1mm] (J_118)[below of=J_116, fill=black] {$$};
   \node[state, minimum size=1mm] (J_119)[below right of=J_117, fill=black] {$$};
  \node[state, minimum size=1mm] (J_120)[below right of=J_118, fill=black] {$$};
 \node[state, minimum size=1mm] (J_121)[above right of=J_117, fill=black] {$$};
 \node[state, minimum size=1mm] (J_122)[below right of=J_120, fill=black] {$$};
 \node(L_1)[below of=J_120] { $\small{H_{26}}$};
  \path (J_116) edge (J_117);
 \path (J_116) edge (J_118);
 \path (J_118) edge (J_120);
 \path (J_120) edge (J_119);
 \path (J_119) edge (J_117);
\path (J_121) edge (J_117);
\path (J_120) edge (J_122);
\path (J_117) edge (J_118);
\path (J_117) edge (J_120);

\node[state, minimum size=1mm] (J_123)[right of=J_117, fill=black] {$$};
 \node[state, minimum size=1mm] (J_124)[right of=J_123, fill=black] {$$};
  \node[state, minimum size=1mm] (J_125)[below of=J_123, fill=black] {$$};
   \node[state, minimum size=1mm] (J_126)[below right of=J_124, fill=black] {$$};
  \node[state, minimum size=1mm] (J_127)[below right of=J_125, fill=black] {$$};
 \node[state, minimum size=1mm] (J_128)[above right of=J_124, fill=black] {$$};
 \node[state, minimum size=1mm] (J_129)[below left of=J_125, fill=black] {$$};
 \node[state, minimum size=1mm] (J_130)[below left of=J_127, fill=black] {$$};
  \node(L_1)[below of=J_127] { $\small{H_{27}}$};
  \path (J_123) edge (J_124);
 \path (J_123) edge (J_125);
 \path (J_125) edge (J_127);
 \path (J_127) edge (J_126);
 \path (J_126) edge (J_124);
\path (J_124) edge (J_125);
\path (J_124) edge (J_127);
\path (J_124) edge (J_128);
\path (J_125) edge (J_129);
\path (J_127) edge (J_130);

\node[state, minimum size=1mm] (J_131)[right of=J_124, fill=black] {$$};
 \node[state, minimum size=1mm] (J_132)[right of=J_131, fill=black] {$$};
  \node[state, minimum size=1mm] (J_133)[below of=J_131, fill=black] {$$};
   \node[state, minimum size=1mm] (J_134)[below right of=J_132, fill=black] {$$};
  \node[state, minimum size=1mm] (J_135)[below right of=J_133, fill=black] {$$};
 \node[state, minimum size=1mm] (J_136)[above right of=J_131, fill=black] {$$};
  \node(L_1)[below of=J_135] { $\small{H_{28}}$};
  \path (J_131) edge (J_132);
 \path (J_131) edge (J_133);
 \path (J_133) edge (J_135);
 \path (J_135) edge (J_134);
 \path (J_134) edge (J_132);
\path (J_131) edge (J_136);
\path (J_132) edge (J_133);
\path (J_131) edge (J_135);

\node[state, minimum size=1mm] (J_137)[right of=J_132, fill=black] {$$};
 \node[state, minimum size=1mm] (J_138)[right of=J_137, fill=black] {$$};
  \node[state, minimum size=1mm] (J_139)[below of=J_137, fill=black] {$$};
   \node[state, minimum size=1mm] (J_140)[below right of=J_138, fill=black] {$$};
  \node[state, minimum size=1mm] (J_141)[below right of=J_139, fill=black] {$$};
 \node[state, minimum size=1mm] (J_142)[below of=J_140, fill=black] {$$};
 \node(L_1)[below of=J_141] { $\small{H_{29}}$};
  \path (J_137) edge (J_138);
 \path (J_137) edge (J_139);
 \path (J_139) edge (J_141);
 \path (J_141) edge (J_140);
 \path (J_140) edge (J_138);
\path (J_137) edge (J_141);
\path (J_139) edge (J_138);
\path (J_140) edge (J_142);

\node[state, minimum size=1mm] (J_143)[right of=J_138, fill=black] {$$};
 \node[state, minimum size=1mm] (J_144)[right of=J_143, fill=black] {$$};
  \node[state, minimum size=1mm] (J_145)[below of=J_143, fill=black] {$$};
   \node[state, minimum size=1mm] (J_146)[below right of=J_144, fill=black] {$$};
  \node[state, minimum size=1mm] (J_147)[below right of=J_145, fill=black] {$$};
 \node[state, minimum size=1mm] (J_148)[above right of=J_144, fill=black] {$$};
\node[state, minimum size=1mm] (J_149)[below of=J_146, fill=black] {$$};
 \node(L_1)[below of=J_147] { $\small{H_{30}}$};
  \path (J_143) edge (J_144);
 \path (J_143) edge (J_145);
 \path (J_145) edge (J_147);
 \path (J_147) edge (J_146);
 \path (J_146) edge (J_144);
\path (J_144) edge (J_148);
\path (J_144) edge (J_145);
\path (J_143) edge (J_146);
\path (J_146) edge (J_149);

\end{tikzpicture}

\begin{tikzpicture}[node distance=0.85cm, inner sep=0pt]

\node[state, minimum size=1mm] (J_150)[right of=J_144, fill=black] {$$};
 \node[state, minimum size=1mm] (J_151)[right of=J_150, fill=black] {$$};
  \node[state, minimum size=1mm] (J_152)[below of=J_150, fill=black] {$$};
   \node[state, minimum size=1mm] (J_153)[below right of=J_151, fill=black] {$$};
  \node[state, minimum size=1mm] (J_154)[below right of=J_152, fill=black] {$$};
 \node[state, minimum size=1mm] (J_155)[below of=J_153, fill=black] {$$};
 \node(L_1)[below of=J_154] { $\small{H_{31}}$};
  \path (J_150) edge (J_151);
 \path (J_150) edge (J_152);
 \path (J_152) edge (J_154);
 \path (J_154) edge (J_153);
 \path (J_153) edge (J_151);
\path (J_151) edge (J_152);
\path (J_151) edge (J_154);
\path (J_153) edge (J_155);
\path (J_150) edge (J_153);

\node[state, minimum size=1mm] (J_156)[right of=J_151, fill=black] {$$};
 \node[state, minimum size=1mm] (J_157)[right of=J_156, fill=black] {$$};
  \node[state, minimum size=1mm] (J_158)[below of=J_156, fill=black] {$$};
   \node[state, minimum size=1mm] (J_159)[below right of=J_157, fill=black] {$$};
  \node[state, minimum size=1mm] (J_160)[below right of=J_158, fill=black] {$$};
 \node[state, minimum size=1mm] (J_161)[below left of=J_160, fill=black] {$$};
 \node(L_1)[below of=J_160] { $\small{H_{32}}$};
  \path (J_156) edge (J_157);
 \path (J_156) edge (J_158);
 \path (J_158) edge (J_160);
 \path (J_160) edge (J_159);
 \path (J_159) edge (J_157);
\path (J_157) edge (J_158);
\path (J_157) edge (J_160);
\path (J_156) edge (J_159);
\path (J_160) edge (J_161);

\node[state, minimum size=1mm] (J_162)[right of=J_157, fill=black] {$$};
 \node[state, minimum size=1mm] (J_163)[right of=J_162, fill=black] {$$};
  \node[state, minimum size=1mm] (J_164)[below of=J_162, fill=black] {$$};
   \node[state, minimum size=1mm] (J_165)[below right of=J_163, fill=black] {$$};
  \node[state, minimum size=1mm] (J_166)[below right of=J_164, fill=black] {$$};
 \node[state, minimum size=1mm] (J_167)[above right of=J_163, fill=black] {$$};
 \node(L_1)[below of=J_166] { $\small{H_{33}}$};
  \path (J_162) edge (J_163);
 \path (J_162) edge (J_164);
 \path (J_164) edge (J_166);
 \path (J_166) edge (J_165);
 \path (J_165) edge (J_163);
\path (J_163) edge (J_167);
\path (J_163) edge (J_164);
\path (J_163) edge (J_166);
\path (J_164) edge (J_165);

\node[state, minimum size=1mm] (J_168)[right of=J_163, fill=black] {$$};
 \node[state, minimum size=1mm] (J_169)[right of=J_168, fill=black] {$$};
  \node[state, minimum size=1mm] (J_170)[below of=J_168, fill=black] {$$};
   \node[state, minimum size=1mm] (J_171)[below right of=J_169, fill=black] {$$};
  \node[state, minimum size=1mm] (J_172)[below right of=J_170, fill=black] {$$};
 \node[state, minimum size=1mm] (J_173)[below left of=J_172, fill=black] {$$};
 \node(L_1)[below of=J_172] {$\small{H_{34}}$};
  \path (J_168) edge (J_169);
 \path (J_168) edge (J_170);
 \path (J_170) edge (J_172);
 \path (J_172) edge (J_171);
 \path (J_171) edge (J_169);
\path (J_169) edge (J_170);
\path (J_169) edge (J_172);
\path (J_170) edge (J_171);
\path (J_173) edge (J_172);

\node[state, minimum size=1mm] (J_174)[right of=J_169, fill=black] {$$};
 \node[state, minimum size=1mm] (J_175)[right of=J_174, fill=black] {$$};
  \node[state, minimum size=1mm] (J_176)[below of=J_174, fill=black] {$$};
   \node[state, minimum size=1mm] (J_177)[below right of=J_175, fill=black] {$$};
  \node[state, minimum size=1mm] (J_178)[below right of=J_176, fill=black] {$$};
 \node[state, minimum size=1mm] (J_179)[above left of=J_175, fill=black] {$$};
 \node[state, minimum size=1mm] (J_180)[below of=J_177, fill=black] {$$};
 \node(L_1)[below of=J_178] { $\small{H_{35}}$};
  \path (J_174) edge (J_175);
 \path (J_174) edge (J_176);
 \path (J_176) edge (J_178);
 \path (J_178) edge (J_177);
 \path (J_177) edge (J_175);
\path (J_175) edge (J_179);
\path (J_177) edge (J_180);
\path (J_176) edge (J_175);
\path (J_177) edge (J_176);
\path (J_175) edge (J_178);

\node[state, minimum size=1mm] (J_181)[right of=J_175, fill=black] {$$};
 \node[state, minimum size=1mm] (J_182)[right of=J_181, fill=black] {$$};
  \node[state, minimum size=1mm] (J_183)[below of=J_181, fill=black] {$$};
   \node[state, minimum size=1mm] (J_184)[below right of=J_182, fill=black] {$$};
  \node[state, minimum size=1mm] (J_185)[below right of=J_183, fill=black] {$$};
 \node[state, minimum size=1mm] (J_186)[above of=J_182, fill=black] {$$};
 \node[state, minimum size=1mm] (J_187)[below left of=J_185, fill=black] {$$};
\node(L_1)[below of=J_185] { $\small{H_{36}}$};
  \path (J_181) edge (J_182);
 \path (J_181) edge (J_183);
 \path (J_183) edge (J_185);
 \path (J_185) edge (J_184);
 \path (J_184) edge (J_182);
\path (J_182) edge (J_183);
\path (J_182) edge (J_185);
\path (J_183) edge (J_184);
\path (J_186) edge (J_182);
\path (J_185) edge (J_187);

\node[state, minimum size=1mm] (J_188)[right of=J_182, fill=black] {$$};
 \node[state, minimum size=1mm] (J_189)[right of=J_188, fill=black] {$$};
  \node[state, minimum size=1mm] (J_190)[below of=J_188, fill=black] {$$};
   \node[state, minimum size=1mm] (J_191)[below right of=J_189, fill=black] {$$};
  \node[state, minimum size=1mm] (J_192)[below right of=J_190, fill=black] {$$};
 \node[state, minimum size=1mm] (J_193)[above of=J_189, fill=black] {$$};
 \node[state, minimum size=1mm] (J_194)[below of=J_190, fill=black] {$$};
\node[state, minimum size=1mm] (J_195)[below of=J_192, fill=black] {$$};
\node(L_1)[below of=J_194] { $\small{H_{37}}$};
  \path (J_188) edge (J_189);
 \path (J_188) edge (J_190);
 \path (J_190) edge (J_192);
 \path (J_192) edge (J_191);
 \path (J_191) edge (J_189);
\path (J_190) edge (J_189);
\path (J_189) edge (J_192);
\path (J_190) edge (J_191);
\path (J_189) edge (J_193);
\path (J_190) edge (J_194);
\path (J_192) edge (J_195);

\node[state, minimum size=1mm] (J_196)[right of=J_189, fill=black] {$$};
 \node[state, minimum size=1mm] (J_197)[right of=J_196, fill=black] {$$};
  \node[state, minimum size=1mm] (J_198)[below of=J_196, fill=black] {$$};
   \node[state, minimum size=1mm] (J_199)[below right of=J_197, fill=black] {$$};
  \node[state, minimum size=1mm] (J_200)[below right of=J_198, fill=black] {$$};
\node(L_1)[below of=J_199] { $\small{H_{38}}$};
  \path (J_196) edge (J_197);
 \path (J_196) edge (J_198);
 \path (J_198) edge (J_200);
 \path (J_200) edge (J_199);
 \path (J_199) edge (J_197);
\path (J_196) edge (J_199);
\path (J_197) edge (J_198);
\path (J_197) edge (J_200);
\path (J_199) edge (J_198);

\end{tikzpicture}

\begin{tikzpicture}[node distance=0.85cm, inner sep=0pt]

\node[state, minimum size=1mm] (J_201)[right of=J_197, fill=black] {$$};
 \node[state, minimum size=1mm] (J_202)[right of=J_201, fill=black] {$$};
  \node[state, minimum size=1mm] (J_203)[below of=J_201, fill=black] {$$};
   \node[state, minimum size=1mm] (J_204)[below right of=J_202, fill=black] {$$};
  \node[state, minimum size=1mm] (J_205)[below right of=J_203, fill=black] {$$};
  \node[state, minimum size=1mm] (J_206)[above of=J_204, fill=black] {$$};
\node(L_1)[right of=J_205] { $\small{H_{39}}$};
  \path (J_201) edge (J_202);
 \path (J_201) edge (J_203);
 \path (J_203) edge (J_205);
 \path (J_205) edge (J_204);
 \path (J_204) edge (J_202);
\path (J_204) edge (J_206);
\path (J_203) edge (J_204);
\path (J_204) edge (J_201);
\path (J_202) edge (J_203);
\path (J_202) edge (J_205);

\node[state, minimum size=1mm] (J_207)[right of=J_202, fill=black] {$$};
 \node[state, minimum size=1mm] (J_208)[right of=J_207, fill=black] {$$};
  \node[state, minimum size=1mm] (J_209)[below of=J_207, fill=black] {$$};
   \node[state, minimum size=1mm] (J_210)[below right of=J_208, fill=black] {$$};
  \node[state, minimum size=1mm] (J_211)[below right of=J_209, fill=black] {$$};

\node(L_1)[below of=J_210] { $\small{H_{40}}$};
  \path (J_207) edge (J_208);
 \path (J_207) edge (J_209);
 \path (J_209) edge (J_211);
 \path (J_211) edge (J_210);
 \path (J_210) edge (J_208);
\path (J_208) edge (J_209);
\path (J_208) edge (J_210);
\path (J_210) edge (J_209);
\path (J_210) edge (J_207);
\path (J_211) edge (J_207);
\path (J_208) edge (J_211);

\end{tikzpicture}

\begin{tikzpicture}[node distance=0.85cm, inner sep=0pt]

\node[state, minimum size=1mm] (J_212)[fill=black] {$$};
 \node[state, minimum size=1mm] (J_213)[below right of=J_212, fill=black] {$$};
  \node[state, minimum size=1mm] (J_214)[below left of=J_212, fill=black] {$$};
   \node[state, minimum size=1mm] (J_215)[below of=J_213, fill=black] {$$};
  \node[state, minimum size=1mm] (J_216)[below of=J_214, fill=black] {$$};
\node[state, minimum size=1mm] (J_217)[below right of=J_214, fill=black] {$$};
\node(L_1)[below of=J_217] { $\small{H_{41}}$};
  \path (J_212) edge (J_213);
 \path (J_212) edge (J_214);
 \path (J_213) edge (J_215);
 \path (J_215) edge (J_217);
 \path (J_217) edge (J_216);
\path (J_216) edge (J_214);
\path (J_213) edge (J_214);

\node[state, minimum size=0mm] (J_1000)[right of=J_213] {$$};

\node[state, minimum size=1mm] (J_7)[above right of=J_1000,fill=black] {$$};
 \node[state, minimum size=1mm] (J_8)[below right of=J_7, fill=black] {$$};
  \node[state, minimum size=1mm] (J_9)[below left of=J_7, fill=black] {$$};
   \node[state, minimum size=1mm] (J_10)[below of=J_8, fill=black] {$$};
  \node[state, minimum size=1mm] (J_11)[below of=J_9, fill=black] {$$};
\node[state, minimum size=1mm] (J_12)[below right of=J_11, fill=black] {$$};
\node(L_1)[right of=J_12] { $\small{H_{42}}$};
  \path (J_7) edge (J_8);
 \path (J_7) edge (J_9);
 \path (J_8) edge (J_10);
 \path (J_10) edge (J_12);
 \path (J_12) edge (J_11);
\path (J_11) edge (J_9);
\path (J_8) edge (J_9);
\path (J_10) edge (J_11);

\node[state, minimum size=0mm] (J_1001)[right of=J_8] {$$};

\node[state, minimum size=1mm] (J_13)[above right of=J_1001, fill=black] {$$};
 \node[state, minimum size=1mm] (J_14)[below right of=J_13, fill=black] {$$};
  \node[state, minimum size=1mm] (J_15)[below left of=J_13, fill=black] {$$};
   \node[state, minimum size=1mm] (J_16)[below of=J_14, fill=black] {$$};
  \node[state, minimum size=1mm] (J_17)[below of=J_15, fill=black] {$$};
\node[state, minimum size=1mm] (J_18)[below right of=J_17, fill=black] {$$};
\node(L_1)[right of=J_18] { $\small{H_{43}}$};
  \path (J_13) edge (J_14);
 \path (J_13) edge (J_15);
 \path (J_15) edge (J_17);
 \path (J_16) edge (J_14);
 \path (J_16) edge (J_18);
\path (J_17) edge (J_18);
 \path (J_14) edge (J_18);
\path (J_14) edge (J_15);

\node[state, minimum size=0mm] (J_1002)[right of=J_14] {$$};

\node[state, minimum size=1mm] (J_19)[above right of=J_1002, fill=black] {$$};
 \node[state, minimum size=1mm] (J_20)[below right of=J_19, fill=black] {$$};
  \node[state, minimum size=1mm] (J_21)[below left of=J_19, fill=black] {$$};
   \node[state, minimum size=1mm] (J_22)[below of=J_20, fill=black] {$$};
  \node[state, minimum size=1mm] (J_23)[below of=J_21, fill=black] {$$};
\node[state, minimum size=1mm] (J_24)[below right of=J_23, fill=black] {$$};
\node(L_1)[right of=J_24] { $\small{H_{44}}$};
  \path (J_19) edge (J_20);
 \path (J_19) edge (J_21);
 \path (J_21) edge (J_23);
 \path (J_23) edge (J_24);
 \path (J_24) edge (J_22);
\path (J_22) edge (J_20);
\path (J_20) edge (J_21);
\path (J_23) edge (J_22);
\path (J_22) edge (J_19);

\node[state, minimum size=0mm] (J_1003)[above right of=J_20] {$$};

\node[state, minimum size=1mm] (J_25)[right of=J_1003, fill=black] {$$};
 \node[state, minimum size=1mm] (J_26)[below right of=J_25, fill=black] {$$};
  \node[state, minimum size=1mm] (J_27)[below left of=J_25, fill=black] {$$};
   \node[state, minimum size=1mm] (J_28)[below of=J_26, fill=black] {$$};
  \node[state, minimum size=1mm] (J_29)[below of=J_27, fill=black] {$$};
\node[state, minimum size=1mm] (J_30)[below right of=J_29, fill=black] {$$};
\node(L_1)[right of=J_30] { $\small{H_{45}}$};
  \path (J_25) edge (J_27);
 \path (J_25) edge (J_26);
 \path (J_27) edge (J_29);
 \path (J_26) edge (J_28);
 \path (J_29) edge (J_30);
\path (J_28) edge (J_30);
\path (J_25) edge (J_28);
\path (J_25) edge (J_29);
\path (J_28) edge (J_29);

\node[state, minimum size=0mm] (J_1004)[above right of=J_26] {$$};

\node[state, minimum size=1mm] (J_31)[right of=J_1004, fill=black] {$$};
 \node[state, minimum size=1mm] (J_32)[below right of=J_31, fill=black] {$$};
  \node[state, minimum size=1mm] (J_33)[below left of=J_31, fill=black] {$$};
   \node[state, minimum size=1mm] (J_34)[below of=J_32, fill=black] {$$};
  \node[state, minimum size=1mm] (J_35)[below of=J_33, fill=black] {$$};
\node[state, minimum size=1mm] (J_36)[below right of=J_35, fill=black] {$$};
\node(L_1)[right of=J_36] { $\small{H_{46}}$};
  \path (J_31) edge (J_32);
 \path (J_31) edge (J_33);
 \path (J_33) edge (J_35);
 \path (J_32) edge (J_34);
 \path (J_35) edge (J_36);
\path (J_36) edge (J_34);
\path (J_32) edge (J_35);
\path (J_34) edge (J_31);
\path (J_31) edge (J_35);

\node[state, minimum size=0mm] (J_1005)[above right of=J_32] {$$};

\node[state, minimum size=1mm] (J_37)[right of=J_1005, fill=black] {$$};
 \node[state, minimum size=1mm] (J_38)[below right of=J_37, fill=black] {$$};
  \node[state, minimum size=1mm] (J_39)[below left of=J_37, fill=black] {$$};
   \node[state, minimum size=1mm] (J_40)[below of=J_38, fill=black] {$$};
  \node[state, minimum size=1mm] (J_41)[below of=J_39, fill=black] {$$};
\node[state, minimum size=1mm] (J_42)[below right of=J_41, fill=black] {$$};
\node(L_1)[right of=J_42] { $\small{H_{47}}$};
  \path (J_37) edge (J_39);
 \path (J_39) edge (J_41);
 \path (J_41) edge (J_42);
 \path (J_42) edge (J_40);
 \path (J_40) edge (J_38);
\path (J_38) edge (J_37);
\path (J_38) edge (J_42);
\path (J_37) edge (J_40);
\path (J_37) edge (J_41);

\end{tikzpicture}

\begin{tikzpicture}[node distance=0.85cm, inner sep=0pt]

\node[state, minimum size=1mm] (J_43)[below of=J_217, fill=black] {$$};
 \node[state, minimum size=1mm] (J_44)[below right of=J_43, fill=black] {$$};
  \node[state, minimum size=1mm] (J_45)[below left of=J_43, fill=black] {$$};
   \node[state, minimum size=1mm] (J_46)[below of=J_44, fill=black] {$$};
  \node[state, minimum size=1mm] (J_47)[below of=J_45, fill=black] {$$};
\node[state, minimum size=1mm] (J_48)[below right of=J_47, fill=black] {$$};
\node(L_1)[right of=J_48] { $\small{H_{48}}$};
  \path (J_43) edge (J_45);
 \path (J_45) edge (J_47);
 \path (J_47) edge (J_48);
 \path (J_48) edge (J_46);
 \path (J_46) edge (J_44);
\path (J_45) edge (J_44);
\path (J_47) edge (J_46);
\path (J_45) edge (J_46);
\path (J_44) edge (J_47);
\path (J_44) edge (J_43);

\node[state, minimum size=0mm] (J_1006)[above right of=J_44] {$$};

\node[state, minimum size=1mm] (J_44)[right of=J_1006, fill=black] {$$};
 \node[state, minimum size=1mm] (J_45)[below right of=J_44, fill=black] {$$};
  \node[state, minimum size=1mm] (J_46)[below left of=J_44, fill=black] {$$};
   \node[state, minimum size=1mm] (J_47)[below of=J_45, fill=black] {$$};
  \node[state, minimum size=1mm] (J_48)[below of=J_46, fill=black] {$$};
\node[state, minimum size=1mm] (J_49)[below right of=J_48, fill=black] {$$};
\node(L_1)[right of=J_49] { $\small{H_{49}}$};
  \path (J_44) edge (J_45);
 \path (J_44) edge (J_46);
 \path (J_46) edge (J_48);
 \path (J_48) edge (J_49);
 \path (J_49) edge (J_47);
\path (J_47) edge (J_45);
\path (J_44) edge (J_47);
\path (J_44) edge (J_49);
\path (J_44) edge (J_48);
\path (J_46) edge (J_45);

\node[state, minimum size=0mm] (J_1007)[above right of=J_45] {$$};

\node[state, minimum size=1mm] (J_50)[right of=J_1007, fill=black] {$$};
 \node[state, minimum size=1mm] (J_51)[below right of=J_50, fill=black] {$$};
  \node[state, minimum size=1mm] (J_52)[below left of=J_50, fill=black] {$$};
   \node[state, minimum size=1mm] (J_53)[below of=J_51, fill=black] {$$};
  \node[state, minimum size=1mm] (J_54)[below of=J_52, fill=black] {$$};
\node[state, minimum size=1mm] (J_55)[below right of=J_54, fill=black] {$$};
\node(L_1)[right of=J_55] { $\small{H_{50}}$};
  \path (J_50) edge (J_52);
 \path (J_50) edge (J_51);
 \path (J_52) edge (J_54);
 \path (J_54) edge (J_55);
 \path (J_55) edge (J_53);
\path (J_53) edge (J_51);
\path (J_50) edge (J_54);
\path (J_50) edge (J_55);
\path (J_50) edge (J_53);
\path (J_50) edge (J_55);
\path (J_54) edge (J_53);

\node[state, minimum size=0mm] (J_1008)[above right of=J_51] {$$};

\node[state, minimum size=1mm] (J_56)[right of=J_1008, fill=black] {$$};
 \node[state, minimum size=1mm] (J_57)[below right of=J_56, fill=black] {$$};
  \node[state, minimum size=1mm] (J_58)[below left of=J_56, fill=black] {$$};
   \node[state, minimum size=1mm] (J_59)[below of=J_57, fill=black] {$$};
  \node[state, minimum size=1mm] (J_60)[below of=J_58, fill=black] {$$};
\node[state, minimum size=1mm] (J_61)[below right of=J_60, fill=black] {$$};
\node(L_1)[right of=J_61] { $\small{H_{51}}$};
  \path (J_56) edge (J_57);
 \path (J_56) edge (J_58);
 \path (J_58) edge (J_60);
 \path (J_60) edge (J_61);
 \path (J_61) edge (J_59);
\path (J_59) edge (J_57);
\path (J_56) edge (J_59);
\path (J_56) edge (J_61);
\path (J_56) edge (J_60);
\path (J_59) edge (J_58);

\node[state, minimum size=0mm] (J_1009)[above right of=J_57] {$$};

\node[state, minimum size=1mm] (J_62)[right of=J_1009, fill=black] {$$};
 \node[state, minimum size=1mm] (J_63)[below right of=J_62, fill=black] {$$};
  \node[state, minimum size=1mm] (J_64)[below left of=J_62, fill=black] {$$};
   \node[state, minimum size=1mm] (J_65)[below of=J_63, fill=black] {$$};
  \node[state, minimum size=1mm] (J_66)[below of=J_64, fill=black] {$$};
\node[state, minimum size=1mm] (J_67)[below right of=J_66, fill=black] {$$};
\node(L_1)[right of=J_67] { $\small{H_{52}}$};
  \path (J_62) edge (J_64);
 \path (J_62) edge (J_63);
 \path (J_64) edge (J_66);
 \path (J_63) edge (J_65);
 \path (J_66) edge (J_67);
\path (J_65) edge (J_67);
\path (J_63) edge (J_64);
\path (J_65) edge (J_66);
\path (J_65) edge (J_62);
\path (J_63) edge (J_67);

\node[state, minimum size=0mm] (J_1010)[above right of=J_63] {$$};

\node[state, minimum size=1mm] (J_68)[right of=J_1010, fill=black] {$$};
 \node[state, minimum size=1mm] (J_69)[below right of=J_68, fill=black] {$$};
  \node[state, minimum size=1mm] (J_70)[below left of=J_68, fill=black] {$$};
   \node[state, minimum size=1mm] (J_71)[below of=J_69, fill=black] {$$};
  \node[state, minimum size=1mm] (J_72)[below of=J_70, fill=black] {$$};
\node[state, minimum size=1mm] (J_73)[below right of=J_72, fill=black] {$$};
\node(L_1)[right of=J_73] { $\small{H_{53}}$};
  \path (J_68) edge (J_69);
 \path (J_68) edge (J_70);
 \path (J_70) edge (J_72);
 \path (J_69) edge (J_71);
 \path (J_72) edge (J_73);
\path (J_71) edge (J_73);
\path (J_68) edge (J_71);
\path (J_68) edge (J_72);
\path (J_69) edge (J_72);
\path (J_69) edge (J_73);

\node[state, minimum size=0mm] (J_1011)[above right of=J_69] {$$};

\node[state, minimum size=1mm] (J_74)[right of=J_1011, fill=black] {$$};
 \node[state, minimum size=1mm] (J_75)[below right of=J_74, fill=black] {$$};
  \node[state, minimum size=1mm] (J_76)[below left of=J_74, fill=black] {$$};
   \node[state, minimum size=1mm] (J_77)[below of=J_75, fill=black] {$$};
  \node[state, minimum size=1mm] (J_78)[below of=J_76, fill=black] {$$};
\node[state, minimum size=1mm] (J_79)[below right of=J_78, fill=black] {$$};
\node(L_1)[right of=J_79] { $\small{H_{54}}$};
  \path (J_74) edge (J_76);
 \path (J_74) edge (J_75);
 \path (J_76) edge (J_78);
 \path (J_78) edge (J_79);
 \path (J_79) edge (J_77);
\path (J_77) edge (J_75);
\path (J_74) edge (J_77);
\path (J_75) edge (J_78);
\path (J_77) edge (J_76);
\path (J_75) edge (J_79);

\end{tikzpicture}
\end{figure}
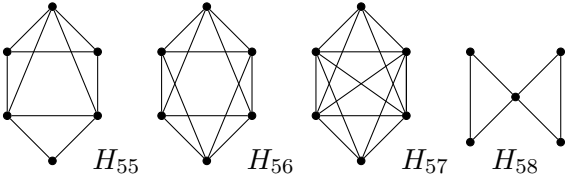
\begin{figure}[!htp]
\begin{tikzpicture}[node distance=0.85cm, inner sep=0pt]
\node[state, minimum size=1mm] (J_80)[below of=J_48,fill=black] {$$};
 \node[state, minimum size=1mm] (J_81)[below right of=J_80, fill=black] {$$};
  \node[state, minimum size=1mm] (J_82)[below left of=J_80, fill=black] {$$};
   \node[state, minimum size=1mm] (J_83)[below of=J_81, fill=black] {$$};
  \node[state, minimum size=1mm] (J_84)[below of=J_82, fill=black] {$$};
\node[state, minimum size=1mm] (J_85)[below right of=J_84, fill=black] {$$};
\node(L_1)[right of=J_85] { $\small{H_{55}}$};
  \path (J_80) edge (J_81);
 \path (J_80) edge (J_82);
 \path (J_82) edge (J_84);
 \path (J_84) edge (J_85);
 \path (J_85) edge (J_83);
\path (J_83) edge (J_81);
\path (J_81) edge (J_82);
\path (J_84) edge (J_83);
\path (J_80) edge (J_84);
\path (J_80) edge (J_83);

\node[state, minimum size=0mm] (J_1012)[above right of=J_81] {$$};

\node[state, minimum size=1mm] (J_86)[right of=J_1012,fill=black] {$$};
 \node[state, minimum size=1mm] (J_87)[below right of=J_86, fill=black] {$$};
  \node[state, minimum size=1mm] (J_88)[below left of=J_86, fill=black] {$$};
   \node[state, minimum size=1mm] (J_89)[below of=J_87, fill=black] {$$};
  \node[state, minimum size=1mm] (J_90)[below of=J_88, fill=black] {$$};
\node[state, minimum size=1mm] (J_91)[below right of=J_90, fill=black] {$$};
\node(L_1)[right of=J_91] { $\small{H_{56}}$};
  \path (J_86) edge (J_87);
 \path (J_86) edge (J_88);
 \path (J_88) edge (J_90);
 \path (J_90) edge (J_91);
 \path (J_91) edge (J_89);
\path (J_89) edge (J_87);
\path (J_88) edge (J_87);
\path (J_89) edge (J_90);
\path (J_86) edge (J_89);
\path (J_86) edge (J_90);
\path (J_88) edge (J_91);
\path (J_87) edge (J_91);

\node[state, minimum size=0mm] (J_1013)[above right of=J_87] {$$};

\node[state, minimum size=1mm] (J_92)[right of=J_1013, fill=black] {$$};
 \node[state, minimum size=1mm] (J_93)[below right of=J_92, fill=black] {$$};
  \node[state, minimum size=1mm] (J_94)[below left of=J_92, fill=black] {$$};
   \node[state, minimum size=1mm] (J_95)[below of=J_93, fill=black] {$$};
  \node[state, minimum size=1mm] (J_96)[below of=J_94, fill=black] {$$};
\node[state, minimum size=1mm] (J_97)[below right of=J_96, fill=black] {$$};
\node(L_1)[right of=J_97] { $\small{H_{57}}$};
  \path (J_92) edge (J_93);
 \path (J_92) edge (J_94);
 \path (J_94) edge (J_96);
 \path (J_96) edge (J_97);
 \path (J_97) edge (J_95);
\path (J_95) edge (J_93);
\path (J_92) edge (J_95);
\path (J_92) edge (J_96);
\path (J_93) edge (J_94);
\path (J_93) edge (J_96);
\path (J_93) edge (J_97);
\path (J_93) edge (J_95);
\path (J_94) edge (J_95);
\path (J_96) edge (J_95);
\path (J_94) edge (J_97);

\node[state, minimum size=1mm] (J_98)[right of=J_93, fill=black] {$$};
 \node[state, minimum size=1mm] (J_99)[below right of=J_98, fill=black] {$$};
  \node[state, minimum size=1mm] (J_102)[below left of=J_99, fill=black] {$$};
   \node[state, minimum size=1mm] (J_100)[above right of=J_99, fill=black] {$$};
  \node[state, minimum size=1mm] (J_101)[below right of=J_99, fill=black] {$$};
\node(L_1)[below of=J_99] { $\small{H_{58}}$};
  \path (J_99) edge (J_98);
 \path (J_99) edge (J_102);
 \path (J_99) edge (J_100);
 \path (J_99) edge (J_101);
 \path (J_98) edge (J_102);
\path (J_100) edge (J_101);
\end{tikzpicture}
\begin{center}
\caption{ Graphs in the $\Psi$ }
\label{f6}  
\end{center}
\end{figure}
\begin{theorem}
\label{the-m}
 Let $G$ be a connected graph of order $n$ and size $m$ which is neither a tree nor a unicyclic graph.
 Then $EC(G)=m$ if and only if $G\in \Psi$.
\end{theorem}
\begin{proof}
Suppose that $EC(G)=m$, and let $l$ be the length of a longest path $P$ in $G$.
We first prove that $l\le6$.

Suppose, to the contrary, that $l\ge7$, and let $P=e_1e_2\ldots e_k$, where $k\ge7$.
Since $EC(G)=m$, the singleton edge partition $X=\{\{e_1\},\{e_2\},\ldots,\{e_m\}\}$
is an $ec$-partition of $G$.

Hence, every singleton edge must form an edge coalition with another singleton edge.
In particular, there exists an edge $x$ such that $\{e_1\}\cup\{x\}$ is an edge-dominating set.

Therefore every edge of $G$ is adjacent to either $e_1$ or $x$.
Writing $e_1=ab$ and $x=zt$, and considering the neighboring edges
$e=ca$, $f=bd$, $h=dz$, and $k=tu$, we see that a path of length at least $7$ necessarily contains an edge that is adjacent to neither $e_1$ nor $x$, a contradiction. Hence $l\le6$.

We now consider the possible values of $l$.
If $l=6$, then a case-by-case analysis shows that the only connected graphs satisfying $EC(G)=m$ are
$A=\{H_{25},H_{27},H_{30},H_{35},H_{36},H_{37}\}.$

If $l=5$, then the corresponding graphs are
\[B=\{H_{6},H_{8},H_{9},H_{10},H_{11},H_{19},\ldots,H_{24}, H_{26},H_{28},H_{29},H_{31},H_{32},H_{33},H_{34},
H_{38},\ldots,H_{57}\}.\]

If $l=4$, then the corresponding graphs are
$C=\{H_{3},H_{4},H_{5},H_{7},H_{12},\ldots,H_{18},H_{58}\}.$

Finally, if $l=3$, then the only possibilities are $D=\{H_1,H_2\}.$

Consequently,  $\Psi=A\cup B\cup C\cup D,$  and therefore  $G\in\Psi$.

The converse is verified directly from the definition of an $ec$-partition.
Hence $EC(G)=m$ if and only if $G\in\Psi$.
\end{proof}
\section{Edge Coalition Graphs of Selected Graph Classes}
In this section, we characterize the edge coalition graphs of certain complete bipartite graphs, double stars, 
wheel graphs, stars, and unicyclic graphs.
\begin{theorem}
\label{the-bipart}
There exist only finitely many edge coalition graphs of the complete bipartite graphs $K_{2,s}$ with $s\ge4$, each of which is isomorphic to one of $K_{l,l'}$, $K_{l,l'}+e$ ($2\le l,l'\le s$), or $K_4$.
\end{theorem}
\begin{proof}
We prove the result for the case $s=4$. The remaining cases follow by the same argument.
To illustrate the proof, we refer to Figure~\ref{f7}.
\begin{figure}[!htp]
\centering
 \begin{tikzpicture}[node distance=2cm, inner sep=0pt, auto]

 \node[state, minimum size=1.5mm] (J_1)[fill=black] {$$};
 \node[state, minimum size=1.5mm] (J_2)[right of=J_1, fill=black] {$$};
 \node[state, minimum size=1.5mm] (J_3)[below left of=J_1, fill=black] {$$};
  \node[state, minimum size=1.5mm] (J_4)[right of=J_3, fill=black] {$$};
  \node[state, minimum size=1.5mm] (J_5)[below right of=J_2, fill=black] {$$};
  \node[state, minimum size=1.5mm] (J_6)[left of=J_5, fill=black] {$$};

  \path (J_1) edge node{$\ a $} (J_3);
  \path (J_1) edge node{$b$} (J_4);
  \path (J_1) edge node{$\ d$} (J_5);
 \path (J_1) edge node{$\ c$}(J_6);
\path (J_2) edge node{$\ e$}(J_3);
 \path (J_2) edge node{$f$}(J_4);
  \path (J_2) edge node{$\ h$}(J_5);
 \path (J_2) edge node{$g$}(J_6);
 \end{tikzpicture}
\caption{Graph $K_{2,4}$}  
\label{f7}
\end{figure}
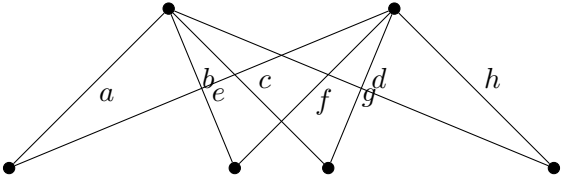

Using Figure~\ref{f7}, we obtain the following $ec$-partitions of  the graph $K_{2,4}$:
\noindent

 $\pi_1=\{\{a\},\{b\},\{c\},\{d\},\{e\},\{f\},\{g\},\{h\}\}$, $\pi_2=\{\{a,b\},\{c\},\{d\},\{e\},\{f\},\{g\},\{h\}\}$,\\
$\pi_3=\{\{a,b,c\},\{d\},\{e\},\{f\},\{g\},\{h\}\}$, $\pi_4=\{\{a,b\},\{c\},\{d\},\{e,f\},\{g\},\{h\}\}$,\\
$\pi_5=\{\{a,b\},\{c,d\},\{e,f\},\{g\},\{h\}\}$,  and $\pi_6=\{\{a,b\},\{c,d\},\{e,f\},\{g,h\}\}$.

Consequently, the corresponding edge coalition graphs are

\noindent 
$ECG(K_{2,4},\pi_1)\simeq K_{4,4}$, $ECG(K_{2,4},\pi_2)\simeq K_{3,4}$, $ECG(K_{2,4},\pi_3)\simeq K_{2,4}+e$,\\ $ECG(K_{2,4},\pi_4)\simeq K_{3,3}$, $ECG(K_{2,4},\pi_5)\simeq K_{2,3}+e$ and $ECG(K_{2,4},\pi_6)\simeq K_4$.

The same construction extends naturally to the case $s\ge5$. 

Let $X=\{a,b\}$ and $Y=\{c_1,\ldots,c_s\}$ be the two partite sets of $K_{2,s}$.
Let $Z$ denote the set of edges incident with $a$, and let $W$ denote the set of edges incident with $b$.
We first construct all non-trivial partitions of $Z$, then partition $W$ accordingly.
Combining these partitions produces all possible $ec$-partitions of $K_{2,s}$.
A systematic analysis of the resulting partitions shows that every associated edge coalition graph is isomorphic to one of
$K_{l,l'}$, $K_{l,l'}+e$, or $K_4$. This completes the proof.
\end{proof}
\begin{theorem}
\label{the-star}
For any star $K_{1,n-1}$ of order $n\ge2$ and any $ec$-partition $\pi$ of $K_{1,n-1}$,
\[ECG(K_{1,n-1},\pi)\simeq (n-1)K_1.\]
\end{theorem}

\begin{proof}
Since every edge of a star is a full edge, every singleton edge forms an edge-dominating set.
Hence the singleton partition $\pi_1$ is the unique $ec$-partition of $K_{1,n-1}$.
Since $\pi_1$ consists of $n-1$ singleton sets and no two singleton sets form an edge coalition, the graph
$ECG(K_{1,n-1},\pi_1)$ has $n-1$ isolated vertices.
Therefore, \[ECG(K_{1,n-1},\pi_1)\simeq (n-1)K_1.\]
\end{proof}
\begin{theorem}
\label{}
If $G$ is a connected graph of order $n$ and size $m=n$, and $\pi_0$ is a singleton $ec$-partition of $G$, then
$ECG(G,\pi_0)$ is isomorphic to one of the graphs in the set
$\Delta=\{3K_1,\;3K_2,\;K_4,\;C_5,\;2K_1\bigcup K_{1,t-1}\;(t\geq3),\;K_{2,3},\;K_{t,s}\;(t,s\geq3),\;S_1\bigcup K_1,\;S_2,\;S_3,\;S_4,\;S_5\}$.
where the graphs $S_1,\ldots,S_5$ are shown in Figure~\ref{f8}.
\begin{figure}[!htbp]
 \begin{tikzpicture}[node distance=0.9cm, inner sep=0pt, auto]

 \node[state, minimum size=1mm] (J_1)[fill=black] {$$};
 \node[state, minimum size=1mm] (J_2)[right of=J_1, fill=black] {$$};
 \node[state, minimum size=1mm] (J_3)[right of=J_2, fill=black] {$$};
  \node[state, minimum size=1mm] (J_4)[below of=J_1, fill=black] {$$};
  \node[state, minimum size=1mm] (J_5)[below  of=J_2, fill=black] {$$};
  \node[state, minimum size=1mm] (J_6)[below of=J_3, fill=black] {$$};
 \node[state, minimum size=1mm] (J_7)[above right of=J_2, fill=black] {$$};
 \node[state, minimum size=1mm] (J_8)[below right of=J_5, fill=black] {$$};
 \node(L_1)[below of=J_5] { $\small{S_{1}}$};
  \path (J_1) edge (J_4);
  \path (J_1) edge (J_5);
  \path (J_1) edge (J_6);
 \path (J_2) edge(J_4);
\path (J_2) edge (J_5);
 \path (J_2) edge (J_6);
  \path (J_3) edge(J_4);
 \path (J_3) edge(J_5);
  \path (J_3) edge(J_6);
 \path (J_2) edge(J_7);
   \path (J_5) edge(J_8);
   \path (J_7) edge(J_8);
  \path (J_7) edge(J_1);
    \path (J_7) edge(J_3);
    \path (J_8) edge(J_4);
    \path (J_8) edge(J_6);

 \node[state, minimum size=0mm] (J_9)[right of=J_3, fill=white] {$$};
 \node[state, minimum size=1mm] (J_10)[right of=J_9, fill=black] {$$};
 \node[state, minimum size=1mm] (J_11)[below left of=J_10, fill=black] {$$};
  \node[state, minimum size=1mm] (J_12)[below right of=J_10, fill=black] {$$};
  \node[state, minimum size=1mm] (J_13)[above left of=J_10, fill=black] {$$};
  \node[state, minimum size=1mm] (J_14)[above right of=J_10, fill=black] {$$};
 \node[state, minimum size=1mm] (J_15)[below of=J_11, fill=black] {$$};
 \node[state, minimum size=1mm] (J_16)[below left of=J_11, fill=black] {$$};
  \node[state, minimum size=1mm] (J_17)[below right of=J_12, fill=black] {$$};
 \node[state, minimum size=1mm] (J_18)[below of=J_12, fill=black] {$$};
 \node(L_1)[below right of=J_11] { $\small{S_{2}}$};
  \path (J_10) edge (J_11);
  \path (J_10) edge (J_12);
  \path (J_11) edge (J_12);
 \path (J_13) edge(J_10);
\path (J_14) edge (J_10);
 \path (J_15) edge (J_11);
  \path (J_16) edge(J_11);
 \path (J_17) edge(J_12);
  \path (J_18) edge(J_12);

  \node[state, minimum size=0mm] (J_19)[right of=J_10, fill=white] {$$};
 \node[state, minimum size=1mm] (J_20)[right of=J_19, fill=black] {$$};
 \node[state, minimum size=1mm] (J_21)[right of=J_20, fill=black] {$$};
  \node[state, minimum size=1mm] (J_22)[right of=J_21, fill=black] {$$};
  \node[state, minimum size=1mm] (J_23)[right of=J_22, fill=black] {$$};
  \node[state, minimum size=1mm] (J_24)[below of=J_20, fill=black] {$$};
 \node[state, minimum size=1mm] (J_25)[below of=J_21, fill=black] {$$};
 \node[state, minimum size=1mm] (J_26)[below right of=J_25, fill=black] {$$};
 \node(L_1)[below of=J_25] { $\small{S_{3}}$};
  \path (J_20) edge (J_24);
  \path (J_20) edge (J_25);
  \path (J_21) edge (J_24);
 \path (J_21) edge(J_25);
\path (J_22) edge (J_24);
 \path (J_22) edge (J_25);
  \path (J_23) edge(J_25);
 \path (J_23) edge(J_26);
  \path (J_25) edge(J_26);
   \path (J_24) edge(J_26);

  \node[state, minimum size=1mm] (J_27)[right of=J_23, fill=black] {$$};
 \node[state, minimum size=1mm] (J_28)[right of=J_27, fill=black] {$$};
 \node[state, minimum size=1mm] (J_29)[right of=J_28, fill=black] {$$};
  \node[state, minimum size=1mm] (J_30)[right of=J_29, fill=black] {$$};
  \node[state, minimum size=1mm] (J_31)[right of=J_30, fill=black] {$$};
  \node[state, minimum size=1mm] (J_32)[below of=J_27, fill=black] {$$};
 \node[state, minimum size=1mm] (J_33)[below of=J_29, fill=black] {$$};
 \node[state, minimum size=1mm] (J_34)[below of=J_31, fill=black] {$$};
 \node(L_1)[below of=J_33] { $\small{S_{4}}$};
  \path (J_32) edge (J_27);
  \path (J_32) edge (J_28);
  \path (J_32) edge (J_31);
 \path (J_33) edge(J_29);
\path (J_33) edge (J_30);
 \path (J_33) edge (J_31);
  \path (J_34) edge(J_31);
  \path (J_32) edge(J_33);

 \node[state, minimum size=0mm] (J_36)[below of=J_4, fill=white] {$$};
 \node[state, minimum size=1mm] (J_37)[below of=J_36, fill=black] {$$};
  \node[state, minimum size=1mm] (J_38)[right of=J_37, fill=black] {$$};
  \node[state, minimum size=1mm] (J_39)[right of=J_38, fill=black] {$$};
  \node[state, minimum size=1mm] (J_40)[right of=J_39, fill=black] {$$};
 \node[state, minimum size=1mm] (J_41)[below of=J_37, fill=black] {$$};
 \node[state, minimum size=1mm] (J_42)[below of=J_38, fill=black] {$$};
  \node[state, minimum size=1mm] (J_43)[below of=J_39, fill=black] {$$};
 \node(L_1)[below of=J_42] { $\small{S_{5}}$};
  \path (J_37) edge (J_42);
  \path (J_38) edge (J_42);
  \path (J_39) edge (J_42);
 \path (J_40) edge(J_42);
\path (J_40) edge (J_41);
 \path (J_39) edge (J_43);
 \end{tikzpicture}
\caption{ Some graphs in the $\Delta$ }
\label{f8}   
\end{figure}
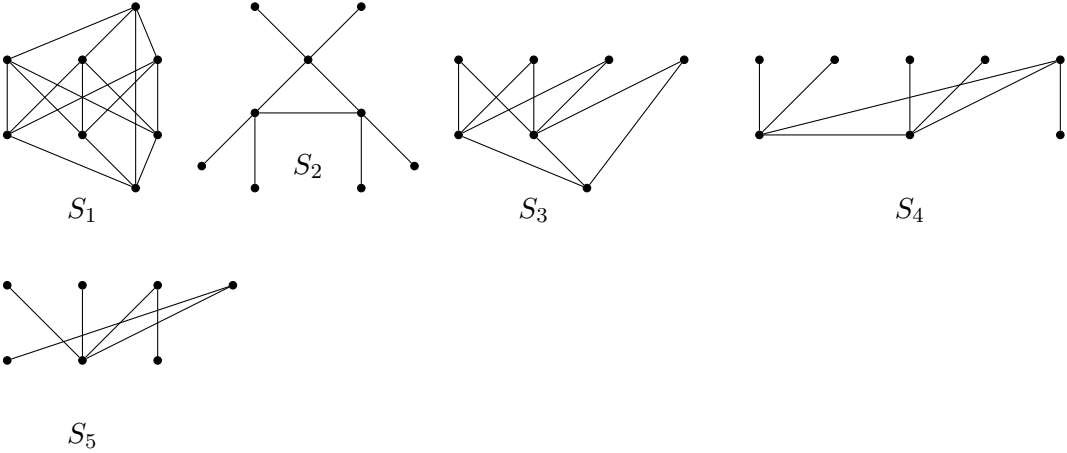
\end{theorem}
\begin{proof}
Since the order and size of the connected graph $G$ are equal, the graph $G$ is unicyclic.

Moreover, since $\pi_0$ is the singleton $ec$-partition of $G$,
From the proof of Theorem~\ref{the-char-n}, it follows that,
 if $l$ denotes the length of the longest path in the graph, then $l\leq6$.
Thus, the graph $G$ satisfying these conditions belongs to the family of graphs illustrated in Figure~\ref{f2},
denoted by $\Theta$, where\\
$\Theta=\{G_{1}, G_{2}, G_{3}, G_{4}, G_{5}, G_{6}, G_{7}, G_{8}, G_{9}, G_{10}, G_{11}, G_{12}\}$.
Since $G$ belongs to the family $\Theta$, it suffices to examine each graph in $\Theta$ individually.
In every case, the corresponding edge coalition graph is isomorphic to one of the graphs in $\Delta$.
\end{proof}
A graph $G$ is called a self-edge coalition graph if 
$G$ is isomorphic to $ECG(G,\pi_0)$ for the singleton $ec$-partition $\pi_0$ of $G$.
The previous results immediately yield the following characterization of self-edge coalition graphs.
\begin{theorem}
Let $G$ be a connected graph of order $n$ and size $m$. Then $G$ is a self-edge coalition graph if and only if
$G\in\{G_{3}, G_{7}\}$.
\end{theorem}
\begin{proof}
If $G\cong G_{3}$ or $G\cong G_{7}$, then it is straightforward to verify that $G$ is a self-edge coalition graph.
Conversely, suppose that $G$ is a self-edge coalition graph.
Since $G\cong ECG(G,\pi_0)$, the graph $ECG(G,\pi_0)$ has exactly $m$ vertices, where $m=|E(G)|$.
Hence $|V(G)|=|E(G)|$, that is, $n=m$.
Thus, from the proof of the previous theorem, only the graphs in the family $\Theta$ satisfy this condition.
However, through a straightforward inspection, we observe that only the two graphs $G_{3}$ and $G_{7}$ are self-edge coalition graphs.
\end{proof}
%
\section{Open problems}
In this paper, we investigated the concept of edge coalitions in graphs.
We studied edge coalition partitions and established several bounds on the edge coalition
number of graphs.
We conclude the paper with the following open problems arising from this research.

\begin{enumerate}
 \item
Characterize the edge coalition graphs of (i) paths,(ii) cycles, and(iii) trees.
\item 
  A graph $G$ is called a singleton edge coalition graph if its singleton edge partition forms an $ec$-partition.
Characterize all singleton edge coalition graphs.
   \item Determine the computational complexity of deciding whether a graph admits an $ec$-partition of a prescribed order.
\end{enumerate}
\section{Acknowledgements}
The authors are grateful to Professor Doost Ali Mojdeh for introducing the research problem and for his valuable discussions.

\end{document}